\documentclass[10pt]{amsart}
\usepackage{amscd}
\usepackage{amssymb}
\newtheorem{Theorem}{Theorem}[section]
\newtheorem{Lemma}[Theorem]{Lemma}
\newtheorem{Proposition}[Theorem]{Proposition}
\newtheorem{Corollary}[Theorem]{Corollary}

\newtheorem{Remark}[Theorem]{Remark}

\def\colim{\operatorname{colim}}
\def\ot{\otimes}
\def\ep{\varepsilon}

\def\cok{\operatorname{cok}}

\def\coeq{\operatorname{coeq}}
\def\eq{\operatorname{eq}}
\def\Cor{\operatorname{Cor}}
\def\Rad{\operatorname{Rad}}
\def\gr{\operatorname{gr}}
\def\Hom{\operatorname{Hom}}

\def\diag{\operatorname{diag}}
\def\Aut{\operatorname{Aut}}
\def\Coalg{\operatorname{Coalg}}
\def\Alg{\operatorname{Alg}}
\def\Vect{\operatorname{Vect}}
\def\Def{\operatorname{Def}}
\def\res{\operatorname{res}}
\def\obs{\operatorname{obs}}
\def\Lift{\operatorname{Lift}}
\def\co-Lift{\operatorname{co-Lift}}
\def\and{\operatorname{and}}

\def\im{\operatorname{im}}

\def\res{\operatorname{res}}
\def\m{\operatorname{m}}

\def\H{H}

\def\YD{\operatorname{YD}}
\def\cal{\mathcal}
\def\incl{\operatorname{incl}}
\def\proj{\operatorname{proj}}
\def\ad{\operatorname{ad}}
\def\adj{\operatorname{adj}}
\def\coadj{\operatorname{coadj}}
\def\ad{\operatorname{ad}}
\def\coeq{\operatorname{coeq}}
\def\del{\partial}
\def\Ext{\operatorname{Ext}}
\def\Tor{\operatorname{Tor}}
\def\Der{\operatorname{Der}}
\def\im{\operatorname{im}}
\def\genst#1#2{\left\langle\left. #1 \right| #2 \right\rangle}
\def\setst#1#2{\left\{\left. #1 \right| #2 \right\}}
\def\gen#1{\left\langle #1 \right\rangle}
\def\set#1{\left\{ #1 \right\}}

\title{Pointed and copointed Hopf algebras as cocycle deformations}
\author{L. Grunenfelder and M. Mastnak}
\address{Department of Mathematics, The University of British Columbia, Vancouver, BC, V6T 1Z2, Canada}
\email{luzius@math.ubc.ca, luzius@mathstat.dal.ca}
\address{Department of Pure Mathematics, University of Waterloo, Waterloo, On, , Canada}
\email{mastnak@math.uwaterloo.ca}
\thanks{Research supported in part by NSERC}
\date{July, 2007}

\begin{document}
\tolerance=1000 \overfullrule 5pt

\begin{abstract} We show that all finite dimensional pointed Hopf
algebras with the same diagram in the classification scheme of
Andruskiewitsch and Schneider are cocycle deformations of each
other. This is done by giving first a suitable characterization of
such Hopf algebras, which allows for the application of results by
Masuoka about Morita-Takeuchi equivalence and by Schauenburg about
Hopf Galois extensions. The ``infinitesimal" part of the deforming
cocycle and of the deformation determine the deformed
multiplication and can be described explicitly in terms of
Hochschild cohomology. Applications to, and results for copointed
Hopf algebras are also considered.
\end{abstract}

\maketitle
\setcounter{section}{-1}

\section{Introduction}\label{s0}

Finite dimensional pointed Hopf algebras over an algebraically closed
field of characteristic zero, particularly when the group of
points is abelian, have been studied quite extensively with
various methods in \cite{AS, BDG, Gr1, Mu}. The most far
reaching results as yet in this area have been obtained in
\cite{AS}, where a large class of such Hopf algebras are
classified. In the present paper we will show, among other things,
that all Hopf algebras in this class can be
obtained by cocycle deformations. We also consider the ``dual"
case, where the Jacobson radical is a Hopf ideal. In the
non-pointed case, in particular when the coradical is not a Hopf
algebra, very little is known. Few examples occur
of the literature \cite{Ra2, Be}, but no general description
or classification results are available. It is the aim of this
paper to contribute to the construction and classification of such
Hopf algebras, in particular the copointed kind. By a copointed Hopf
algebra we mean a Hopf algebra
$H$ whose Jacobson radical $\Rad H$ is a Hopf ideal and $H/{\Rad
H}$ is a group algebra.

If $H$ is a Hopf algebra with coradical a Hopf subalgebra then the
graded coalgebra $\gr_cH$ associated with the coradical filtration
is a graded Hopf algebra and its elements of positive degree form
the radical. If the radical of $H$ is a Hopf ideal then the graded
algebra associated with the radical filtration is a graded Hopf
algebra with $\Cor (\gr_a H)\cong H/\Rad H$. In either case we
have $\gr H\cong R\# H_0$, where $H_0$ is the degree zero part and
$R$ is the braided Hopf algebra of coinvariants or invariants,
respectively.

The Nichols algebra $B(V)$ of a crossed $kG$-module $V$ is a
connected graded braided Hopf algebra. $H(V)=B(V)\# kG$ is an
ordinary graded Hopf algebra with coradical $kG$ and the elements
of positive degree form a Hopf ideal (the graded radical). A lifting of
$H(V)$ is a pointed Hopf algebra $H$ for which $\gr_cH\cong H(V)$. Such liftings
are obtained by deforming the multiplication of $H(V)$.
The lifting problem for $V$ asks for the classification of all liftings
of $H(V)$. This problem, together with the characterization of $B(V)$ and $H(V)$,
have been solved  by Andruskiewitsch and Schneider in \cite{AS} for a large class of crossed $kG$-modules of finite Cartan type.
It allows them to classify all finite dimensional pointed Hopf algebras
$A$ for which the order of the abelian group of points has no
prime factors $< 11$. In this paper we find a description of these
lifted Hopf algebras, which is suitable for the application of a
result of Masuoka about Morita-Takeuchi equivalence \cite{Ma} and
of Schauenburg about Hopf Galois extensions \cite{Sch}, to prove
that all liftings of a given $H(V)$ in this class are cocycle
deformations of each other. As a result we see here that in the
class of finite dimensional pointed Hopf algebras classified by
Andruskiewitsch an Schneider \cite{AS} all Hopf algebras $H$ with
isomorphic associated graded Hopf algebra $\gr_cH$ are monoidally
Morita-Takeuchi equivalent, and therefore cocycle deformations of
each other. For some special cases such results have been obtained
in \cite{Ma, Di}. Since $B(V)$ and $H(V)$ are graded the cocycles and
deformations can be viewed in a formal setting. The infinitesimal
parts are Hochschild cocycles. They determine the deformed
multiplication and can be computed explicitly. The dual problem is
to construct co-pointed Hopf algebras $H$ by deforming the
comultiplication of $H(V)$ in such a way that $\gr_aH\cong H(V)$,
where $\gr_aH$ is the graded Hopf algebra associated with the
radical filtration. Such `liftings' can be viewed as cocycle
deformations with convolution invertible coalgebra cocycles, which
again can be discussed in a formal setting. In both cases the
deformations are formal in the sense of \cite{GS}, the
infinitesimal parts of these deformations determine the deformed
multiplication and comultiplication, respectively, and are
determined by the $G$-invariant part of the Hochschild cohomology
of $B(V)$ or $H(V)$.

First we discuss what von Neumann regularity for an algebra and
the ``dual" concept of coregularity for a coalgebra  entail in the
case of a Hopf algebra. Of particular interest is the situation
where the coradical is a regular Hopf algebra and/or $H/{\Rad H}$
is a coregular Hopf algebra. If both conditions are satisfied then
$H\cong A\#\Cor (H)$, where $A=H^{co(\Cor H)}$ is the braided Hopf
algebra of coinvariants over $\Cor (H)$. This happens in
particular when $\Cor (H)$ is a finite dimensional Hopf subalgebra
of $H$ and $\Rad (H)$ is a Hopf ideal of finite codimension in
$H$. It also happens for $\gr_cH$ when $H_0=\Cor (H)$ is a regular
Hopf algebra, and for $\gr_aH$ when $H_0=H/\Rad (H)$ is a
coregular Hopf algebra. A group Hopf algebra $kG$ is always
coregular, but it is regular if and only if every finitely
generated subgroup of $G$ is finite. In Section 2 we continue with
a short review of braided spaces, braided Hopf algebras, Nichols
algebras and bosonization. in preparation for a useful
characterization of the liftings for a large class of crossed
modules over finite abelian groups in Section 3. With this
characterization it is then possible to prove that liftings are
Morita-Takeuchi equivalent (Section 3) by using Masuoka's pushout
construction, and that they are cocycle deformations of each other
(Section 4) by a result of Schauenburg.  Cocycle deformations of
multiplication and comultiplication, as well as their relation to
Hochschild cohomology are discussed in Section 4. Some explicit
examples are presented in Section 5, and duality of the two
deformation procedures are explored in the final section.

It came to our attention that, in the preprint \cite{Ma2} just
posted in the archive, some special cases of the connection
between \lq liftings' and cocycle deformations are considered.
After we had posted our paper A. Masuoka informed us that his
Theorem 2 of \cite{Ma} was missing a condition, as observed in
\cite{BDR}. For the verification of this additional condition,
needed in our \ref{main3}, we refer to the second version of
\cite{Ma2}, where it now appears as an appendix.

\section{Regularity and coregularity}\label{s1}

An algebra $A$ is (Von Neumann) regular if $a=axa$ has a solution
for every $a\in A$. This is equivalent to saying that every left
(right) $A$-module is flat, or also that every finitely generated
left (right) ideal of $A$ is generated by an idempotent [C, St],
\cite{We}. We say that a coalgebra $C$ is coregular if every left
(right) $C$-comodule is coflat.

\begin{Lemma}\label{coproducts} If $C=\oplus_{\nu}C_{\nu}$ is a
coalgebra and $X$ a (right) $C$-comodule, then $X=\oplus X_{\nu}$,
where $X_{\nu}=i_{\nu}^*X=X\ot^CC_{\nu}$. Moreover, $X$ is
$C$-coflat if and only if $X_{\nu}$ is $C_{\nu}$-coflat for every
$\nu$.
\end{Lemma}

\begin{Proposition} The following properties of a coalgebra
$C$ are equivalent:
\begin{itemize}
\item[{(a)}] $C$ is coregular.
\item[{(b)}] Every subcoalgebra of $C$ is coregular.
\item[{(c)}] Every finite dimensional subcoalgebra of $C$ is cosemisimple.
\item[{(d)}] $\Cor (C)=C$.
\item[{(e)}] $C^*$ is a regular algebra.
\end{itemize}
\end{Proposition}

\begin{proof} Let $C$ be coregular and let $D$ be a subcoalgebra of $C$.
Every (right) $D$-comodule is also a (right) $C$-comodule, hence
coflat as a $C$-comodule. Now, if $X$ is a right $D$-comodule and
$f\colon M\to N$ is a surjective $D$-comodule map, then in the
commutative diagram
$$\begin{CD}
X\ot^DM @> >> X\ot^CM \\
@VVV @VVV \\
X\ot^DN @> >> X\ot^CN
\end{CD}$$
the horizontal maps are bijective by the definition of the
cotensor product, while the right-hand vertical map is surjective
by the $C$ coflatness of $X$, so that the left-hand vertical map
is also surjective and $X$ is therefore coflat as a $D$-module.
Thus, every (right) $D$ comodule is coflat and so $D$ is
coregular.

If $D$ is a finite dimensional subcoalgebra of $C$, which is
coregular, then $D^*$ is a finite dimensional regular algebra,
hence semisimple, so that $D$ is cosemisimple. (A module is flat
whenever every finitely generated submodule is flat \cite{St}. An
$A$-module $Y$ is flat if and only if $Y\ot_RI\to Y$ is injective
for every finitely generated left ideal $I$ of $A$ [C, St]. Thus,
it suffices to consider finitely generated modules, in which case
$(X\ot^DM)^*\cong X^*\ot_{D^*}M^*$.)

If every finite dimensional subcoalgebra of $C$ is cosemisimple
then $\Cor (C)=C$, since every element of $C$ is contained in a
finite dimensional subcoalgebra.

That (d) implies (a) follows from Lemma \ref{coproducts}. It
remains to show that (d) is equivalent to (e). If $C=\Cor (C)$
then $C^*$ is a product of finite dimensional simple, and hence
regular algebras. But a product of algebras is regular if and only
if each factor is regular, so that $C^*$ is regular. Conversely,
if $C$ is not coregular, then it contains a finite dimensional
subcoalgebra $D$ which is not cosemisimple, and $D^*$ is a finite
dimensional non-regular quotient algebra of $C^*$, so that $C^*$
is not regular.
\end{proof}

\begin{Proposition} If $A$ is a regular algebra then:
\begin{itemize}
\item[{(a)}] $\Rad (A)=0$.
\item[{(b)}] Every quotient algebra of $A$ is regular.
\item[{(c)}] Every finite dimensional quotient algebra of $A$ is semisimple.
\item[{(d)}] $A^{\circ}$ is a coregular coalgebra, i.e. $\Cor (A^{\circ})=A^{\circ}$.
\end{itemize}
\end{Proposition}

\begin{proof} If $a\in\Rad (A)$ then $a=axa$ for some $x\in A$ implies that
$a(1-xa)=0$ and hence $a=0$, since $1-xa$ is invertible.

If $a=axa$ in $A$ then $\bar a\bar x\bar a=\bar a$ in $A/I$ for
any ideal $I$ of $A$. If $A/I$ is finite dimensional and regular then it is
semisimple, since $\Rad (A)=0$.

$A^{\circ}=\colim (A/I)^*$, where the colimit is over all cofinite
ideals of $A$. If $A$ is regular and $I$ is a cofinite ideal then
$A/I$ is semisimple and $(A/I)^*$ is cosemisimple, so that $\Cor
(A^{\circ})=A^{\circ}$. Assertion (d) also appears as Proposition
3.2 in \cite{Cu}.
\end{proof}

\begin{Lemma}\label{l14} If $A$ is a Von Neumann regular subring of the ring $B$ then $A\cap\Rad B=0$.
\end{Lemma}

\begin{proof} If $a\in A\cap\Rad B$ then $a=axa$ has a solution in $A$, say
$x=a'$, since $A$ is Von Neumann regular, and $1-a'a$ is invertible, since
$a\in\Rad (B)$. But then $a(1-a'a)=0$ implies that $a=0$.
\end{proof}

The following example shows that the conclusion of this Lemma does
not hold in general when $A$ is not regular. If the polynomial
algebra $B=k[x]$ is considered in the usual way as a subalgebra of
the power series algebra $A=k[[x]]$ then $\Rad (B)=Bx$ and
$A\cap\Rad (B)=Ax$, but $\Rad (A)=0$.

\begin{Proposition} Let $H$ be a Hopf algebra.
\begin{itemize}
\item[{\rm(a)}] If $\Cor (H)$ is a Von Neumann regular Hopf
subalgebra of $H$, in particular if $\Cor (H)$ is a finite
dimensional Hopf subalgebra of $H$, then the algebra map
$\Cor(H)\to H\to H/\Rad (H)$ is injective.
\item[{\rm(b)}] If $\Rad (H)$ is a Hopf ideal and $H/\Rad (H)$ is coregular,
in particular if $\Rad (H)$ is a Hopf ideal of finite codimension in $H$,
then the coalgebra map $\Cor (H)\to H\to H/\Rad (H)$ is surjective.
\item[{\rm(c)}] If $\Cor (H)$ is a
von Neumann regular Hopf subalgebra of $H$ and $\Rad (H)$ a Hopf
ideal with $H/\Rad (H)$ coregular then $\Cor (H)\to H\to H/\Rad
(H)$ is a Hopf algebra isomorphism and $H\cong A\rtimes\Cor (H)$,
where $A=H^{co(\Cor H)}$ is the braided Hopf algebra of
coinvariants over $\Cor (H)$. This happens in particular when
$\Cor (H)$ is a finite dimensional Hopf subalgebra of $H$ and
$\Rad (H)$ is a Hopf ideal of finite codimension in $H$.
\end{itemize}
\end{Proposition}

\begin{proof} a) If $\Cor (H)$ is a Von Neumann regular Hopf
subalgebra of $H$ then for any $a\in \Cor (H)\cap\Rad
(H)\subset\Cor (H)$ the equation $a=axa$ has a solution in $\Cor
(H)$ so that $a(1-xa)=0$. Since $1-xa$ is invertible in $H$ it
follows that $a=0$, and hence $\Cor (H)\cap\Rad (H)=0$. In
particular, if $\Cor (H)$ is a finite dimensional Hopf subalgebra
of $H$ then it is cosemisimple, hence semisimple by \cite{LR}, and
thus Von Neumann regular, so that $\Cor (H)\cap\Rad
(H)\subseteq\Rad (\Cor (H))=0$ by Lemma \ref{l14}

b) If $\Rad (H)$ is a Hopf ideal in $H$ and $H/\Rad (H)$ is a coregular Hopf
algebra, then $\Cor (H/\Rad (H))=H/\Rad (H)$. Moreover, since any
surjective coalgebra map $\eta\colon C\to D$, where $D=\Cor (D)$,
maps $\Cor C$ onto $\Cor D$ \cite[Corollary 5.3.5]{Mo} it follows
that the canonical map $\Cor (H)\to H\to H/\Rad (H)$ is
surjective. In particular when $\Rad (H)$ is a Hopf ideal of
finite codimension in $H$ then $H/\Rad (H)$ is a semisimple Hopf
algebra, hence cosemisimple by \cite{LR}.

c) It follows directly from a) and b) that $\pi\colon H\to H/\Rad
(H)$ is a Hopf algebra with projection. Now apply \cite{Ra1}.
\end{proof}

\noindent {\bf Pointed and copointed Hopf algebras.} A Hopf
algebra $H$ is pointed if its coradical $\Cor H$ is equal to the
group algebra of the group of points $G(H)$. In this case the
coradical filtration is an ascending Hopf algebra filtration and
the associated graded Hopf algebra $\gr^cH$ has the obvious
injection $\kappa^c\colon kG\to\gr^cH$ and projection $\pi^c\colon
\gr^cH\to kG$ such that $\pi^c\kappa^c =1$.

We say that $H$ is copointed if its radical $\Rad H$ is a Hopf
ideal and $H/{\Rad H}$ is a group algebra $kG$. Here the radical
filtration is an descending Hopf algebra filtration and again the
associated graded Hopf algebra $\gr^rH$ has the obvious projection
$\pi^r\colon \gr^rH\to kG$ and an injection $\kappa^r\colon
kG\to\gr^rH$ such that $\pi^r\kappa^r =1$.

In both cases above $\gr H$ is graded, pointed and copointed, and
by \cite{Ra1} $\gr H\cong A\# kG$, where $A=\setst{ x\in
H}{(\pi\ot 1)\Delta (x) =1\ot x}$ is the graded connected braided
Hopf algebra of coinvariants.

\begin{Lemma}\label{l17} If the Hopf algebra $H$ is pointed and copointed then $\Cor H\cong H/{\Rad H}$ and $H$ is a Hopf algebra with
projection. Moreover, $R\# kG\cong H$., where $R=\setst{x\in \gr
H}{(p\ot 1)\Delta (x)=1\ot x}$ is the connected braided Hopf
algebra of coinvariants of $H$. This is the case in particular for
$\gr^cH$ and for $\gr^rH$ when $H$ is pointed or copointed,
respectively.
\end{Lemma}

\begin{proof} A surjective coalgebra map $\eta\colon C\to D$, where
$D=\Cor (D)$, maps $\Cor C$ onto $\Cor D$ \cite{Mo}. Thus, the
composite $\Cor H\to H\to H/{\Rad H}$ is a bijection. The
isomorphism is that of \cite{Ra1}.
\end{proof}

\section{Braided Hopf algebras and the (bi-)cross
product}\label{s2}

A braided monoidal category $\mathcal V$ is a monoidal category
together with a natural morphism $c\colon V\ot W\to W\ot V$ such
that
\begin{enumerate}
\item $c_{k,V}=\tau =c_{V,k}$,
\item $c_{U\ot V,W}=(c_{U,W}\ot 1)(1\ot c_{V,W})$,
\item $c_{U,V\ot W}=(1\ot c_{U,W})(c_{U,V}\ot 1)$,
\item $c(f\ot g)=(g\ot f)c$.
\end{enumerate}

Braided algebras, braided coalgebras and braided Hopf algebras are
now defined with this tensor product and braiding in mind. The
compatibility condition $\Delta\m =(m\ot m)(1\ot c\ot
1)(\Delta\ot\Delta )$ between multiplication and comultiplication
in a braided Hopf algebra $A$ involves the braiding $c\colon A\ot
A\to A\ot A$, so that the diagram

$$\begin{CD}
A\ot A @> (1\ot c\ot 1)(\Delta\ot\Delta )>> A\ot A\ot A\ot A \\
@V m VV  @V m\ot m VV \\
A @> \Delta >>  A\ot A
\end{CD}$$
commutes, i.e. multiplication and unit are morphisms of braided
coalgebras or, equivalently, comultiplication and counit are maps
of braided algebras.

\subsection{Primitives and indecomposables}

The vector space of primitives
$$P(A) =\setst{ y\in A}{\Delta (y)=y\ot 1+1\ot y}\cong\ker
(\tilde{\Delta})\colon A/k\to A/k\ot A/k$$ of a braided Hopf
algebra $A$ is a braided vector space, since $\Delta$ is  a map in
$\mathcal V$. The c-bracket map $[-,-]_c=m(1\ot 1-c)\colon A\ot
A\to A$ restricted to $P(A)$ satisfies $\Delta [x,y]_c=[x,y]_c\ot
1+(1-c^2)x\ot y +1\ot [x,y]_c$; in particular $[x,y]_c\in P(A)$ if
and only $c^2(x\ot y)=x\ot y$. Moreover, if $x\in P(A)$ and
$c(x\ot x)=qx\ot x$ then $\Delta x^n=\sum_{i+j=n}{n\choose
i}_qx^i\ot x^j$, where ${n\choose
i}_q={\frac{n_q!}{i_q!(n-i)_q!}}$ are the $q$-binomial
coefficients (the Gauss polynomials) for $q$, $m_q!=1_q2_q...m_q$
with $j_q=1+q+...+q^{j-1}$ if $j>0$ and $0_q!=1$. If $q=1$ then
$j_q=j$ and we have the ordinary binomial coefficients, otherwise
$j_q={\frac{1-q^j}{1-q}}$. In particular, if $q$ has order $n$
then ${n\choose i}_q=0$ for $0<i<n$, and hence $x^n\in P(A)$.

\begin{Lemma} Let $\set{x_i}$ be a basis of $P(A)$ such that
$c(x_i\ot x_j)=q_{ji}x_jx_i$. If $q_{ji}q_{ij}q_{ii}^{r-1}=1$ then
$\ad x_i^r(x_j)$ is primitive.
\end{Lemma}

\begin{proof} See for example \cite{AS1} appendix 1.
\end{proof}

The vector space of indecomposables
$$Q(A)=JA/JA^2=\cok (\tilde m\colon JA\ot JA\to JA),$$
where $JA=\ker (\epsilon )$, is in $\cal V_c$. The c-cobracket map
$\delta_c=(1\ot 1-c)\Delta\colon A\to A\ot A$ restricts to $JA$,
since
$$\Delta (\bar a)= \bar a\ot 1+1\ot\bar a +\sum \bar a_i\ot\bar b_i$$
and hence
$$\delta (\bar a)=\sum (\bar a_i\ot\bar b_i -c(\bar a_i\ot\bar b_i))$$
is in $JA\ot JA$ for every $\bar a=a-\epsilon (a)\in JA$. Moreover,
$$\delta (\bar a\bar b)-(\bar a\ot\bar b -c^2(\bar a\ot\bar b))$$ is
in $JA^2\ot JA + JA\ot JA^2$. In particular, if $c^2(\bar a\ot\bar
b)=\bar a\ot\bar b$, then $\delta (\bar a\ot\bar b)\in JA^2\ot
JA+JA\ot JA^2$.

\subsection{The free and the cofree graded braided Hopf algebras}

The forgetful functor $U\colon \Alg_c\to \mathcal V_c$ has a
left-adjoint $\cal A\colon \mathcal V_c\to \Alg_c$ and the
forgetful functor $U\colon \Coalg_c\to\mathcal V_c $ has a
right-adjoint $\cal C\colon \mathcal V_c\to \Coalg_c$, the free
braided graded algebra functor and the cofree graded braided
coalgebra functor, respectively. Moreover, there is a natural
transformation $\cal S\colon \cal A\to\cal C$, the shuffle map or
quantum symmetrizer. They can be described as follows.

If $(V, \mu , \delta )$ is a braided vector space then the tensor
powers $T_0(V)=k$, $T_{n+1}(V)=V\ot T_n(V)$ are braided vector
spaces as well and so is $T(V)=\oplus_nT_n(V)$. The ordinary
tensor algebra structure makes $T(V)$ the free connected graded
braided algebra, and the ordinary tensor coalgebra structure makes
it the cofree  connected graded braided coalgebra.

By the universal property of the graded braided tensor algebra
$T(V)$ the linear map $\Delta_1=\incl\diag\colon V\to T(V)\ot
T(V)$, $\Delta_1 (v)=v\ot 1+1\ot v$, induces the c-shuffle
comultiplication $\Delta_{\cal A}\colon T(V)\to T(V)\ot T(V)$,
which is a homomorphism of braided algebras, so that $\Delta_{\cal
A}m=(m\ot m)(1\ot c\ot 1)(\Delta_{\cal A}\ot\Delta_{\cal A})$.
Moreover, the linear map $s_1\colon V\to T(V)$, $s_1(v)=-v$,
extends uniquely to a c-antipode $s_{\cal A}\colon T(V)\to T(V)$,
such that $s_{\cal A}m=m(s_{\cal A}\ot s_{\cal A})c$,
$\Delta_{\cal A}s_{\cal A}=c(s_{\cal A}\ot s_{\cal A})\Delta_{\cal
A}$ and $m(1\ot s_{\cal A})\Delta_{\cal A}=\iota\epsilon
=m(s_{\cal A}\ot 1)\Delta_{\cal A}$, thus making $\cal A(V)=(T(V),
m,\Delta_{\cal A},s_{\cal A})$ the free connected graded braided
Hopf algebra. This defines a functor $\cal A\colon \mathcal V_c\to
Hopf_c$, left-adjoint to the space of primitives functor $P\colon
Hopf_c\to\mathcal V_c$.

On the other hand, by the universal property of the cofree
connected graded braided coalgebra $T(V)$ there is a unique
c-shuffle multiplication $m_{\cal C}\colon T(V)\ot T(V)\to T(V)$,
which is the homomorphism of braided coalgebras induced by the
linear map $m_1=+\proj\colon T(V)\ot T(V)\to V$, so that $\Delta
m_{\cal C}=(m_{\cal C}\ot m_{\cal C})(1\ot c\ot 1)(\Delta\ot\Delta
)$. The linear map $s_1=-\proj\colon  T(V)\to V$ induces uniquely
a c-antipode $s_{\cal C}\colon T(V)\to T(V)$, such that $s_{\cal
C}m_{\cal C}=m_{\cal C}(s_{\cal C}\ot s_{\cal C})c$, $\Delta
s_{\cal C}=c(s_{\cal C}\ot s_{\cal C})\Delta$ and $m_{\cal C}(1\ot
s_{\cal C})\Delta =\iota\epsilon =m_{\cal C}(s_{\cal C}\ot
1)\Delta$, making $\cal C(V)=(T(V), \Delta , m_{\cal C}, s_{\cal
C})$ the cofree connected graded braided Hopf algebra. The functor
$\cal C \colon \mathcal V_c\to Hopf_c$ is right-adjoint to the
space of indecomposables functor $Q\colon Hopf_c\to\mathcal V_c$.

The adjunctions just described provide natural isomorphisms
$$\mathcal{V}_c(Q\cal A(V),V)\cong Hopf_c(\cal A(V),
\cal C(V))\cong \cal V_c(V,P\cal C(V)),$$
and by construction we also have natural isomorphisms
$$Q\cal A (V)\cong V \quad , \quad V\cong P\cal C (V).$$
The resulting natural isomorphism
$$Vect_c(V,V)\cong Hopf_c(\cal A(V),\cal C(V))$$
sends the identity morphism of $V$ to the quantum symmetrizer
$\cal S\colon \cal A(V)\to\cal C(V)$. The image of $\cal S$ is the
Nichols algebra
$$\cal B(V)\cong\cal A(V)/\ker (\cal S)\cong\im\cal S\subset\cal C(V)$$
and $Q\cal B(V)\cong V\cong P\cal B(V)$.

An explicit description of the quantum symmetrizer can be obtained
in term of the action of the braid groups $B_n$ on the tensor
powers $V^{\ot n}$ as follows. The Braid group $B_n$ can be
defined by generators $\sigma_1, \sigma_2, \ldots , \sigma_{n-1}$
and relations
\begin{enumerate}
\item $\sigma_i\sigma_j=\sigma_j\sigma_i$ for $|i-j| > 1$ and
\item $\sigma_i\sigma_{i+1}\sigma_i=\sigma_{i+1}\sigma_i\sigma_{i+1}$.
\end{enumerate}
The symmetric group $S_n$ is obtained by imposing the additional relations
\begin{enumerate}
\item $\sigma_i^2=1$ for $i=1, 2, \ldots , n-1$.
\end{enumerate}
If we denote the corresponding generators of $S_n$ by $\tau_1,
\tau_2, \ldots , \tau_{n-1}$, then the kernel of the canonical
quotient map $f\colon B_n\to S_n$ given by $f(\sigma_i)=\tau_i$ is
the normal subgroup of $B_n$ generated by the squares $\sigma_1^2,
\sigma_2^2, \ldots , \sigma_{n-1}^2$. The set theoretic section
$u\colon S_n\to B_n$ defined by
$u(\tau_{i_1}\tau_{i_2}\cdots\tau_{i_l})=
\sigma_{i_1}\sigma_{i_2}\cdots\sigma_{i_l}$
for any reduced word $\tau_{i_1}\tau_{i_2}\cdots\tau_{i_l}$ of
$S_n$ is called the Matsumoto section. If $l(\tau\tau
')=l(\tau)+l(\tau ')$ then $u(\tau\tau ')=u(\tau )u(\tau ")$. The
element $\cal S=\sum_{\tau\in S_n}u(\tau )$ of $kB_n$ is called
the quantum symmetrizer.

The braiding map $c\colon V\ot V\to V\ot V$ induces a linear
representation $\rho_n\colon B_n\to \Aut (V^{\ot n})$ by $\rho_n
(\sigma_i)=1^{\ot (i-1)}\ot c\ot 1^{\ot (n-i-1)}$ for every $n\ge
0$, and hence a graded linear map $\cal S\colon T(V)\to T(V)$. If
$X_{i,j}$ is the subset of $(i,j)$-shuffles in $S_n$ then $\cal
S_{i,j}=\sum_{\tau\in X_{i,j}}u(\tau )$ is an element of $kB_n$,
and $\Delta_{\cal A}(n)=\sum_{i+j=n}\cal S_{i,j}$ and $m_{\cal
C}(i,j)=\cal S_{i,j}$. Moreover, since $\cal S_{i,j}(\cal
S_i\ot\cal S_j)=\cal S_n$ whenever $r+s=n$, it follows that the
quantum symmetrizer actually induces homomorphism of graded
braided Hopf algebras
$$\cal S\colon \cal A(V)\to\cal C(V),$$
also called quantum symmetrizer.

The free graded braided Hopf algebra $\cal A(V)$ is the  graded
braided tensoralgebra $T(V)$, with the graded braided c-shuffle
comultiplication $\Delta_{\cal A}\colon \cal A(V)\to \cal
A(V)\ot\cal A(V)$ and antipode $s_{\cal A}\colon \cal A(V)\to\cal
A(V)$, induced by the  universal property of $T(V)$ from the
natural maps $\Delta_1\colon V\to T(V)\ot T(V)$, $\Delta_1
(v)=v\ot 1+1\ot v$, and $s_1\colon V\to T(V)$, $s_1(v)=-v$,
respectively.

$$\begin{CD}
\cal A(V)\ot\cal A(V) @> (1\ot c\ot 1)(\Delta_{sh}\ot\Delta_{sh})>>
\cal A(V)\ot\cal A(V)\ot\cal A(V)\ot\cal A(V) \\
@VmVV  @Vm\ot mVV\\
\cal A(V) @>\Delta_{sh} >> \cal A(V)\ot\cal A(V)
\end{CD}$$
commute. This defines a functor $\cal A\colon \mathcal V_c\to
Hopf_c$, left-adjoint to the space of primitives functor $P\colon
Hopf_c\to\mathcal V_c$.

The cofree graded braided Hopf algebra $\cal C(V)$ is the graded
braided tensor coalgebra with the graded braided tensor coalgebra
structure induced by the canonical decompositions
$\Delta_{i,j}\colon T_n(V)\to T_i(V)\ot T_j(V)$ for $n=i+j$, and
the graded braided c-shuffle multiplication induced by the
c-shuffles $m_{i,j}\colon T_i(V)\ot T_j(V)\to T_{i+j}(V)$, which
make the diagram

$$\begin{CD}
\cal C(V)\ot\cal C(V) @> (1\ot c\ot 1)(\Delta\ot\Delta )>>
\cal C(V)\ot\cal C(V)\ot\cal C(V)\ot\cal C(V) \\
@Vm_{sh}VV  @Vm_{sh}\ot m_{sh}VV\\
\cal C(V) @>\Delta >> \cal C(V)\ot\cal C(V)
\end{CD}$$
commute. The functor $\cal C\colon \mathcal V_c\to Hopf_c$ is
right-adjoint to the space of indecomposables functor $Q\colon
Hopf_c\to\mathcal V_c$.

\subsection{Crossed modules} A prime example of a
braided monoidal category is the category of crossed
$H$-modules $\YD^H_H$ for a Hopf algebra $H$. A crossed $H$-module
or a Yetter-Drinfel'd $H$-module, $(V,\mu ,\delta )$ is a vector
space $V$ with a $H$-module structure $\mu \colon H\ot V\to V$,
$\mu (h\ot v)=hv$, and a $H$-comodule structure $\delta\colon V\to
H\ot V$ , $\delta (v)=v_{-1}\ot v_0$, such that $h\delta
(v)=h_1v_{-1}\ot h_2v_0=(h_1v)_{-1}h_2\ot (h_1v)_0$, or
$$(m\ot\mu )(1\ot\tau\ot 1)(\Delta\ot\delta )=
(m\ot 1)(1\ot\tau )(\delta\mu\ot 1)(1\ot\tau )(\Delta\ot 1),$$
i.e such that the diagram

$$\begin{CD}
H\ot V @>\Delta\ot\delta >> H\ot H\ot H\ot V @>1\ot\tau\ot 1 >> H\ot H\ot H\ot V \\
@V(1\ot\tau )(\Delta\ot 1)VV  @.  @Vm\ot\mu VV \\
H\ot V\ot H @>\delta\mu\ot 1 >> H\ot V\ot H @>(m\ot 1)(1\ot\tau )>> H\ot V
\end{CD}$$
commutes. This is the case in particular when
$\delta (hv)=h_1v_{-1}s(h_3)\ot h_2v_0$, i.e: when the diagram

$$\begin{CD}
H\ot V @> (1\ot\phi\ot 1)(\Delta\ot\delta ) >> H\ot H\ot H\ot V\\
@V\mu VV  @V m\ot\mu VV \\
V @> \delta >> H\ot V
\end{CD}$$
commutes, where $\phi = (m\ot 1)(1\ot s\ot 1)(1\ot\tau\Delta
)\tau\colon H\ot H\to H\ot H$, $\phi (g\ot h)=hs(g_2)\ot g_1$. The
braided $H$-modules with the obvious homomorphisms form a braided
monoidal category, with the ordinary tensor product of vector
spaces together with diagonal action and diagonal coaction. The
braiding, given by $c(v\ot w)=v_{-1}(w)\ot v_0$,
$$c=(\mu\ot 1)(1\ot\tau )(\delta\ot 1)\colon V\ot W\to W\ot V,$$
clearly satisfies the braiding conditions. The crossed $H$-module
$(k, \mu =\ep\ot 1, \delta =\iota\ot 1)$ acts as a unit for the
tensor. Moreover, $(H,\adj ,\Delta )$ and $(H, m, \coadj )$ are
crossed $H$-modules, where $\adj (h\ot h')=h_1h'S(h_2)$ and
$\coadj (h)=h_1S(h_3)\ot h_2$.

\subsection{The pushout construction for bi-cross products} Recall Masuoka's pushout construction for Hopf
algebras \cite{Ma}, \cite{Gr1}. If $A$ is a Hopf algebra then
$\Alg (A,k)$  s a group under convolution which acts on $A$ by
conjugation as Hopf algebra automorphisms.

\begin{Lemma} For every Hopf algebra $A$ the group $\Alg (A,k)$ acts on $A$ by conjugation' as Hopf algebra automorphisms
$$\rho\colon \Alg(A,k)\to\Aut_{Hopf}(A),$$
where $\rho_f=f*1*fs$, i.e: $\rho_f(x)=f(x_1)x_2f(sx_3)$. The
image of $\rho$ is a normal subgroup of $\Aut_{Hopf}(A)$.
\end{Lemma}

\begin{proof} It is easy to verify that $\rho_f$ is an Hopf algebra
map. The definition of $\rho$ shows
that $\rho_{f_1*f_2}=f_1*f_2*1*f_2s*f_1s=\rho_{f_1}\rho_{f_2}$
and, since $f*fs=\ep =fs*f$, it follows that
$\rho_f\rho_{fs}=1=\rho_{fs}\rho_f$, so that $\rho_f$ is a Hopf
algebra automorphism. If $\phi\in\Aut_{Hopf}(A)$ and
$f\in\Alg(A,k)$ then $\phi^{-1}\rho_f\phi
=f\phi*1*fs\phi=\rho_{f\phi}$, hence the image of $\rho$ is a
normal subgroup.
\end{proof}

Two Hopf ideals $I$ and $J$ of $A$ are said to be conjugate if
$J=\rho_f(I)=f*I*fs$ for some $f\in\Alg (A,k)$. If $x\in P_{1,g}$
is a $(1,g)$-primitive then
$$\rho_f(x)=f(x)+f(g)x+f(g)gfs(x)=f(g)x+f(x)(g-1)].$$

\begin{Theorem}\cite[Theorem 2]{Ma}\cite[Theorem 3.4]{BDR}\label{po} Let $A'$ be a Hopf
subalgebra of $A$. If the Hopf ideals $I$ and $J$ of $A'$ are
conjugate and $A/(f*I)\ne 0$ then the quotient Hopf algebras $A/(I)$ and $A/(J)$ by the Hopf ideals in
$A$ generated by $I$ and $J$ are monoidally Morita-Takeuchi equivalent,
i.e: there exists a $k$-linear monoidal equivalence
between their (left) comodule categories.
\end{Theorem}

\begin{proof} With the correction that $A/(f*I)\ne 0$ (\cite{BDR}, Theorem 3.4), Masuoka's result (\cite{Ma}, Theorem 2) that there is a $(A/(I),A/(J))$-biGalois object, namely $A/(f*I)$, holds , and we can invoke
\cite[Corollary 5.7]{Sch}, to see that $A/(I)$ and $A/(J)$ are
Morita-Takeuchi equivalent.
\end{proof}

Observe, as Masuoka did \cite{Ma}, that the commutative square
$$\begin{CD}
A' @>>> A \\
@VVV @VVV \\
B/I @>>> A/(I)
\end{CD}$$
is a pushout of Hopf algebras.

If $R$ is a braided Hopf algebra in the braided category of
crossed $H$-modules then the bi-cross product $R\# H$ is an
ordinary Hopf algebra with multiplication
$$(x\# h)(x'\# h') = xh_1(x')\# h_2h'$$
and comultiplication
$$\Delta (x\# h) = x_1\# (x_2)_{-1}h_1\ot (x_2)_0\# h_2.$$
The (left) action of $H$ on $R$ induces a (right) action on
$\Alg(R,k)$ by $fh(x)=f(hx)$. An algebra map $f\colon R\to k$ is
$H$-invariant if $fh=\ep (h)f$ for all $h\in H$.

\begin{Proposition}\label{invariant} Let $K$ be a
Hopf algebra in the braided category of crossed $H$-modules and
let $\Alg_H(K,k)$ be the set of $H$-invariant algebra maps. Then:
\begin{enumerate}
\item $\Alg_H(K,k)$ is a group under convolution.
\item The restriction map
$\res\colon  _H\Alg_H(K\# H,k)\to\Alg_H(K,k)$, $\res (F)=F\ot\iota$,
is an isomorphism of groups with inverse given by $\res^{-1}(f)=f\ot\ep$.
\item The image of the conjugation homomorphism
$$\Theta =\rho\res^{-1}\colon \Alg_H (K,k)\to \Aut_{Hopf}(K\# H)$$
is contained in $\widetilde{\Aut}_{Hopf}(K\# H)= \setst{
\phi\in\Aut_{Hopf}(K\# H)}{\phi_{|H}=id}$.
\end{enumerate}
\end{Proposition}

\begin{proof} The set of algebra maps $\Alg (K,k)$ may
not be a group, but since the coequalizer
$K^H=\coeq (\mu ,\ep\ot 1\colon H\ot K\to K$ is an ordinary Hopf
algebra, $\Alg_H(K,k)\cong\Alg(K^H,k)$ is a group under
convolution. More directly, if $f, f'\in \Alg_H(K,k)$ then
\begin{eqnarray*}
f*f'(x y) & = & f\ot f'(x_1(x_2)_{-1}y_1\ot (x_2)_0y_2) = f(x_1)f'(y_1)f(x_2)f'(y_2) \\
         & = & (f*f')(x)(f*f')(y)  \\
f*f'(h x) & = & f\ot f'(h_1x_1\ot h_2x_2)=f(h_1x_1)f'(h_2x_2)=\ep
(h)(f*f')(x),
\end{eqnarray*}
and $f*fs  = \ep =fs*f,$ so that $\Alg_H(K,k)$ is closed under
convolution multiplication and inversion.

For $F\in {_H{\Alg_H(K\# H,k)}}$ the map $\res(F)\colon K\to k$ is
in fact a $H$-invariant algebra map, since
\begin{eqnarray*}
\res(F)(h x)&=& F(hx\ot 1) = F((1\ot h_1)(x\ot 1)(1\ot s(h_3)) \\
&=& \ep (h)F(x\ot 1)=\ep (h)\res (F)(x)
\end{eqnarray*}
and
\begin{eqnarray*}
\res (F)(x y)=F(xy\ot 1)=F(x\ot 1)F(y\ot 1)=res(F)(x)\res (F)(y).
\end{eqnarray*}
If $F'\in _H{\Alg_H(K\# H,k)}$ as well, then
\begin{eqnarray*}
\res(F*F')(x) &=& F\ot F'(x_1\ot (x_2)_{-1}\ot (x_2)_0\ot 1) \\
&=& F(x_1\ot 1)F'(x_2\ot 1)=\res (F)*\res (F')(x),
\end{eqnarray*}
showing that $\res$ is a group homomorphism. It is now easy to see
that $\res$ is invertible and that the inverse is as stated.

As a composite of two group homomorphisms $\Theta$ is obviously a
group homomorphism. Moreover,
$$\Theta (f)(1\ot h)=\res^{-1}(f)*1*\res^{-1}(f)s(1\ot h)=
\ep (h_1)(1\ot h_2)\ep (h_3)=1\ot h$$
for $f\in\Alg_H(K,k)$, showing that $\Theta (f)_{|H}=id$.
\end{proof}

\begin{Corollary} Let $R$ be a braided Hopf algebra in the braided
category of crossed $H$-modules and let $K$ be a braided Hopf
subalgebra. If $I$ is a Hopf ideal in $K$ and $f\in\Alg_H(K,k)$
then,
\begin{itemize}
\item $J=I\# H$ and $J_f=\Theta (f)(J)$ are Hopf ideals in $K\# H$,
\item $R\# H/(J)=R/(I)\# H$ and $R\# H/(J_f)$ are monoidally
Morita-Takeuchi equivalent, if $(R\# H)/(\res^{-1}(f)*J)\ne 0$.
\end{itemize}
\end{Corollary}

\begin{Proposition} Let $K$ be a (braided) Hopf subalgebra of the
(braided) Hopf algebra $R$ such that $R$ is
left or right faithfully flat over $K$, and let $B=R/RK^+R$,
where $RK^+R$ is the Hopf ideal of $R$ generated by $K^+$. Then:
\begin{enumerate}
\item $B$ is a (braided) Hopf algebra and
$$\begin{CD}
K @>\kappa >> R \\
@V\ep VV  @V\pi VV \\
k @>\iota >> B
\end{CD}$$
is a pushout diagram of (braided) Hopf algebras.
\item $R^{coB}\cong K \cong {^{coB}R}$.
\item The commutative square
$$\begin{CD}
K @>
\kappa >> R \\
@V\ep VV  @V\pi VV \\
k @>\iota >> B
\end{CD}$$
is a pullback of (braided) Hopf algebras.
\end{enumerate}
\end{Proposition}

\begin{proof} Since $K^+=\ker (\ep )$ is a Hopf ideal in $K$ it
follows that $RK^+R$ is a Hopf ideal in $R$ and hence that
$B=R/RK^+R$ is a (braided) Hopf algebra quotient of $R$.

First show that
$$K {\buildrel\kappa\over\longrightarrow} R
{{\buildrel f\over\longrightarrow}
\atop{\buildrel g\over\longrightarrow}} R\ot_KR,$$
where $f=(1\ot\iota\ep )\Delta$, $f(r)=r\ot 1$, and
$g=(\iota\ep\ot 1)\Delta$, $g(r)= r\ot 1$, is exact, i.e: an
equalizer diagram by faithful flatness (\cite{Wa}, Theorem 13.1).
Observe that both $f$ and $g$ are injective maps. Consider the
diagram
$$M {\buildrel w\over\longrightarrow} M\ot_KR
{{\buildrel u\over\longrightarrow}
\atop{\buildrel v\over\longrightarrow}} M\ot_KR\ot_KR$$
with $w=1\ot\iota$, $u=1\ot f$ and $v=1\ot g$. Then $w(m)=m\ot 1$,
$u(m\ot r)=m\ot r\ot 1$, $v(m\ot r)=m\ot 1\ot r$ and $uw=vw$. If
$N=\eq (u,v)$ then $M$ is contained in $N$ and the diagram
$$N\ot_KR \to M\ot_KR\ot_KR {{\buildrel {u\ot 1}\over\longrightarrow}
\atop{\buildrel {v\ot 1}\over\longrightarrow}} M\ot_K\ot_KR\ot_KR\ot_KR$$
is exact by the flatness of $R$ over $K$. Define $s\colon
M\ot_K\ot_KR\ot_KR\ot_KR\to M\ot_KR\ot_KR$ by $s=1\ot 1\ot m$,
$s(m\ot r\ot r'\ot r'')=m\ot r\ot r'r''$. If $x=\sum m_i\ot r_i\ot
r'_i\in \eq (u\ot 1,v\ot 1)=N\ot_KR$ then
$$\sum m_i\ot r_i\ot 1\ot r'_i=(u\ot 1)x =(v\ot 1)x
=\sum m_i\ot 1\ot r_i\ot r'_i $$ and hence
$$x=s(u\ot 1)x=s(v\ot 1)=\sum m_i\ot 1\ot r_ir'_i=
(w\ot 1)(\sum m_i\ot r_ir'_i)\in\im (M\ot_KR)$$
which shows that $M\ot_KP=N\ot_KR=\eq (u\ot 1,v\ot 1)$ and
$N/M\ot_KR\cong n\ot_KR/M\ot_KR=0$. This implies that $N/M=0$ by
faithful flatness of $R$ over $K$, and hence that $M=N$. In
particular, if $M=K$ this gives
$$K=\eq (u,v)=\eq (f,g).$$
Now if $\pi\colon R\to B=R/RK^+R$ is the canonical projection then
$$R^{coB}\to R {{\buildrel\phi\over\longrightarrow}
\atop{\buildrel\psi\over\longrightarrow}} R\ot B,$$
where $\phi =(1\ot\iota_B\ep_R)\Delta_R$, $\phi (x)=x\ot 1$, and
$\psi =(1\ot\pi )\Delta_R$, $\psi (x)=x_1\ot\pi (x_2)$, is an
exact equalizer diagram by definition. The Galois map $\beta
=(m_R\ot\pi )(1\ot\Delta_R)\colon R\ot_KR\to R\ot B$, defined by
$\beta (x\ot y)=xy_1\ot\pi (y_2)$, has inverse $\beta^{-1}\colon
R\ot B\to R\ot_KR$ given by $\beta^{-1}(1\ot p)=(m_R\ot 1)(1\ot
s\ot 1)(1\ot\Delta_R)$, $\beta^{-1}(x\ot\pi (y))=xs(y_1)\ot y_2$.
The diagram
$$\begin{CD}
K @>>> R @>f-g>> R\ot_KR \\
@.    @|   @V\beta VV   \\
R^{coB} @>>> R @>\phi -\psi>>  R\ot B \\
\end{CD}$$
commutes, so that $K\cong R^{coB}$. With the same argument, but on
the left, one gets $K\cong {^{coB}R}$.

For any commutative diagram of (braided) Hopf algebras
$$\begin{CD}
X @>\alpha >> R \\
@V\ep_X VV @V\pi VV \\
k @>\iota_B >> B
\end{CD}$$
which is commutative, $\pi\alpha =\iota_B\ep_X$, we get
$$\phi\alpha =(1\ot\iota_B\ep_R)\Delta_R\alpha =
(\alpha \ot\iota_B\ep_X)\Delta_X =(1\ot\pi )\Delta_R\alpha
=\Psi\alpha .$$
Thus, since
$$K\to R {{\buildrel\phi\over\longrightarrow}
\atop{\buildrel\psi\over\longrightarrow}} R\ot B$$
is an equalizer diagram, there is a unique Hopf algebra map
$\gamma\colon X\to K$ such that $\kappa\gamma =\alpha$ and
$\ep_K\gamma =\ep_X$, which shows that the diagram in item (3) is
a pullback of (braided) Hopf algebras.
\end{proof}

Observe that a Hopf algebra with cocommutative coradical is
faithfully flat over any of its Hopf subalgebra \cite{Ta}.

\subsection{An exact 5-term sequence} Let $ad_l\colon R\ot R\to R$
and $ad_r\colon R\ot R\to R$ be the left and the
right adjoint actions, that is $ad_l(r\ot r')=r_1a's(r_2)$ and
$ad_r(r'\ot r)=s(r_1)r'r_2$, respectively. A Hopf subalgebra $K$
of the Hopf algebra $R$ is said to be normal if it is stable under
the left and the right adjoint action of $R$.

\begin{Lemma} Let $K$ be a Hopf subalgebra of $R$.
\begin{enumerate}
\item If $K$ is normal in $R$, then $RK^+=K^+R$ is a
Hopf ideal of $R$, $R/RK^+\cong R\ot_Kk$ is a Hopf algebra,
$\pi\colon R\to R/RK^+$ is Hopf algebra map and
$$\begin{CD}
K @>\kappa >> R \\
@V\ep VV  @V\pi VV \\
k @>\iota >> R/RK^+
\end{CD}$$
is a pushout of Hopf algebras.
\item If $\pi R\to R'$ is a Hopf algebra map then
$$R^{coR'}=\eq (\Delta (1\ot\iota\ep ),
\Delta (1\ot\pi )\colon R\to R\ot R'$$
is stable under the left and the right adjoint action and
$$\begin{CD}
R^{coR'} @>\kappa >> R \\
@V\ep VV @V\pi VV \\
k @>\iota >> R'
\end{CD}$$
is a pullback.
\end{enumerate}
\end{Lemma}

\begin{proof} If $x\in R$ and $y\in K^+$ then $xy=ad_l(x_1\ot y)x_2$ and
$yx=x_1ad_r(y\ot x_2)$. Since $K$ is normal in $R$ it follows that
$RK^+=K^+R$ and $I=RK^+$ is an ideal. It is a Hopf ideal, since
$RK^+$ is always a coideal and since $s(I)=I$. Observe that
$$R\ot K {{\buildrel{m(1\ot\kappa )}\over\longrightarrow}
\atop{\buildrel {\ep\ot 1}\over\longrightarrow}}
R \buildrel \pi\over\longrightarrow R/RK^+$$
is a coequalizer. If $u\colon R\to X$ is a Hopf algebra map such
that $u\kappa =\iota\ep $, then $um_R(1\ot\kappa )=m_X(u\ot
u)(\kappa\ot 1)=m_X(\iota\ep\ot u)\ep\ot u$. Thus, there is a
unique Hopf algebra map $u'\colon R/RK^+\to X$ such that $u'\pi
=u$ and $u'\iota =\iota_X$, showing that the diagram in (1) is a
pushout.

If $y\in R^{coR'}$ then $ad_l(x\ot y)=x_1ys(x_2)$ and $(1\ot\pi
)\Delta ad_l(x\ot y)=x_1y_1s(x_4)\ot\pi
(x_2y_2s(x_3))=x_1ys(x_2)\ot 1=(1\ot\ep )\Delta ad_l(x\ot y)$ and
similarly for $ad_r(y\ot x)$. The diagram obviously commutes. If
$v\colon Z\to A$ is a Hopf algebra map such that $\pi
v=\iota_{R'}\ep_Z$ then $(1\ot\pi )\Delta v=(v\ot\pi v)\Delta
=(v\ot\iota_{R'}\ep_Z)\Delta = (1\ot\iota\ep )(v\ot v)\Delta
=(1\ot\iota\ep )\Delta v$. Hence there is a unique Hopf algebra
map $v'\colon Z\to R^{coR'}$ such that $\kappa v'=\ep v'$, and the
diagram in (2) is a pullback.
\end{proof}

\begin{Proposition} Let $K$ be a (braided) Hopf subalgebra of $R$
such that $R$ is left or right faithfully flat over $K$,
and such that $RK^+=K^+R$. If $B=R/RK^+$. Then:
\begin{enumerate}
\item $R^{coB}=K= {^{coB}R}$.
\item $K$ is a normal Hopf subalgebra of $R$.
\item There are spectral sequences
$$H^p(B,H^q(K,Y))^{p\atop{\Longrightarrow}}H^n(R,Y),$$
$$H_p(B,H_q(K,X))^{p\atop{\Longrightarrow}}H_n(R,X).$$
\item Exact sequences in low degrees
$$0\to H^1(B,Y)\to H^1(R,Y)\to Hom_B(K^+\ot_Kk,Y)\to H^2(B,Y)\to H^2(R,Y),$$
$$H_2(R,X)\to H_2(B,X)\to (K^+\ot_Kk)\ot_BX\to H_1(R,X)\to H_1(B,X)\to 0.$$
\end{enumerate}
\end{Proposition}

\begin{proof} For items (1) 1nd (2) see \cite{Mo}, Proposition 3.4.3. The spectral sequences are special cases of those for normal subalgebras
\cite{CE}, Chap. XVI, Theorem 6.1, where it actually suffices to
assume that $R$ is flat as a $K$-module. The natural isomorphisms
$H^p(R,Y)\cong\Ext^p_R(k,Y)$ and $H_p(R,X)\cong \Tor^R_1(k,X)$,
valid for any Hopf algebra $R$, have been used. Finally, the low
degree 5-term sequences of these spectral sequences are those of
\cite{CE}, page 329, cases C and C', where the isomorphisms
$$E^{01}_2=\Hom_B(k,H^1(K,Y))\cong Hom_B(k,\Hom_K(K^+,Y))\cong \Hom_B(K^+\ot_Kk,Y)$$
and
$$E^2_{01}= H_1(K,k)\ot_BX\cong (K^+\ot_Kk)\ot_BX$$
have been used.
\end{proof}

The 5-term exact sequences can also be found directly without the
use of spectral sequences. The exact sequence of $K$-modules $0\to
K^+\to K\to k\to 0$ induces an exact sequence of $R$-modules $0\to
K^+\ot_KR\to K\ot_KR\to k\ot_KR\to 0$, that is $K^+\ot_KR=K^+R$
and $k\ot_KR=R/K^+R$, since $R$ is $K$-flat. Again, since $R$ is
$K$-flat, any $R$-projective (or $R$-flat) resolution $\mathbf X$
of an $R$-module $M$ is also a $K$-flat resolution of $M$. For
every injective $K$-module map $f\colon Y\to Y'$ the $R$-module
map $1\ot f:R\ot_KY\to R\ot_KY'$ is injective, since $R$ is
$K$-flat, and thus $f\ot_K1:Y\ot_KX\cong (Y\ot_RR)\ot_KX\to
(Y'\ot_RR)\ot_KX\cong Y'\ot_KX$ is injective as well for every
flat $R$-module $X$. This gives an isomorphism of complexes
$$k\ot_K\mathbf X\cong k\ot_KR\ot_R\mathbf X\cong B\ot_R\mathbf X,$$
so that $H_n(K,M)=\Tor_n^K(k,M)\cong\Tor_n^R(B,M)$. From the exact
sequence of $R$-modules $0\to R^+\to R\to k\to 0$ we then get a
commutative diagram of $B$-modules with exact rows {\small
$$\begin{CD}
0 @>>> \Tor_1^R(B,k) @>>> B\ot_RR^+ @>>> B\ot_RR @>>> B\ot_Rk @>>> 0 \\
@. @V\cong VV @| @V\cong VV @V\cong VV @. \\
0 @>>> H_1(K,k) @>>> B\ot_RR^+ @>>> B @>>> k @>>> 0
\end{CD}$$}
and hence an exact sequence of $B$-modules
$$0\to K^+/(K^+)^2\to B\ot_RR^+\to B^+\to 0.$$
Apply the functor $\Hom_B(\ ,M)$ to this last exact sequence to get
\begin{eqnarray*}
0&\to& \Hom_B(B^+,M)\to \Hom_B(B\ot_RR^+,M)\to\Hom_B(H_1(K,k),M) \\
&\to&\Ext_B^1(B^+,M)\to \Ext_B^1(B\ot_RR^+,M)
\end{eqnarray*}
If the bottom sequence of the commutative diagram
$$\begin{CD}
0 @>>> M @>>> P @>\tilde p >> R^+ @>>> 0 \\
@.  @|  @V\beta VV  @V\alpha VV @. \\
0 @>>> M @>>> Y @>p >> B\ot_RR^+ @>>> 0
\end{CD}$$
is an extension $E$ of $B$-modules then it is also an extension of
$R$-modules. The map $\alpha\colon R^+\to B\ot_RR^+$, given by
$\alpha (x)=1\ot_Rx$, is an $R$-module homomorphism. The right
hand square of the diagram is a pullback and the top sequence is a
representative of $\alpha^*(E)$. If the top sequence is split by
$\tilde u\colon R^+\to P$ then $u\colon B\ot_RR^+\to Y$, given by
$u(b\ot x)=b\beta\tilde u(x)$, splits the bottom sequence, since
$pu(b\ot x)=bp\beta u(x)=b\alpha\tilde pu(x)=b\alpha (x)=b\ot x$.
Hence the induced map $\alpha^*\colon
\Ext^1_B(B\ot_RR^+,M)\to\Ext^1_R(R^+,M)$ is injective. Moreover,
$\Ext^1_R(R^+,M)\cong H^2(R,M)$.

Using $\Hom_B(B\ot_RR^+,M)\cong \Hom_R(R^+,M)$  one gets the
commutative diagram with exact rows and columns
$$\begin{CD}
 @. \Hom_B(B,M) @= \Hom_R(R,M) @.   \\
@.  @Vu^*VV  @Vv^*VV @. \\
0 @>>> \Hom_B(B^+,M) @>\kappa^*>> \Hom_R(R^+,M) @>\delta >> \Hom_B(H_1(K,k),M) \\
@.  @VVV  @V\del VV @. \\
0 @>>> \Ext^1_B(k,M) @>>> \Ext^1_R(k,M) @.\\
@. @VVV  @VVV @. \\
 @. 0 @. 0 @.
\end{CD}$$
in which $\delta v^*=\delta\kappa^*u^*=0$, so that there is a
unique homomorphism \begin{eqnarray*}\gamma\colon \Ext^1_R(k,M)\to
\Hom_B(H_1(K,k),M) \end{eqnarray*} such that $\gamma\del =\delta$.
It follows that the sequence
$$0\to H^1(B,M)\to H^1(R,M)^{{\gamma}\atop{\longrightarrow}}
\Hom_B(H_1(K,k),M)\to H^2(B,M)\to H^2(R,M)$$
is exact. Similar argument work for the homology sequence.

\section{Liftings over finite abelian groups}\label{s3}

In this section we give a somewhat different characterization of
the class of finite dimensional pointed Hopf algebras classified
in \cite{AS}, and show that any two such Hopf algebras with
isomorphic associated graded Hopf algebras are monoidally
Morita-Takeuchi equivalent, and therefore cocycle deformations of
each other, as we will point out in the next section.

A datum of finite Cartan type
$$\cal D =\cal D \left(G, (g_i)_{1\le i\le \theta},
(a_{ij})_{1\le i,j\le\Theta}\right)$$ for a (finite) abelian group
$G$ consists of elements $g_i\in G$, $\chi_j\in \widehat{G}$ and a
Cartan matrix $(a_{ij})$ of finite type satisfying the Cartan
condition
$$q_{ij}q_{ji}=q_{ii}^{a_{ij}}$$
with $q_{ii}\ne 1$, where $q_{ij}=\chi_j(g_i)$, in particular
$q_{ii}^{a_{ij}}=q_{jj}^{a_{ji}}$ for all $1\le i,j\le \theta$. In
general, the matrix $(q_{ij}$ of a diagram of Cartan type is not
symmetric, but by \cite[Lemma 1.2]{AS} it can be reduced to the
symmetric case by twisting.

Let $\mathbf Z [I]$ be the free abelian group of rank $\theta$
with basis $\set{\alpha_1, \alpha_2,\ldots ,\alpha_{\theta}}$. The
Weyl group $W\subset \Aut (\mathbf Z[I])$ of $(a_{ij})$ is
generated by the reflections $s_i\colon \mathbf Z[I]\to \mathbf
Z[I]$, where $s_i(\alpha_j)=\alpha_j -a_{ij}\alpha_i$ for all
$i,j$. The root system of the Cartan matrix $(a_{ij})$ is $\Phi
=\cup_{i=1}^{\theta}W(\alpha_i)$ and $\Phi^+ =\Phi\cap\mathbf
[I]=\setst{\alpha\in\Phi}{\alpha =\sum _{i=1}^{\theta}n_i\alpha_i
, n_i\ge 0}$ is the set of positive roots relative to the basis of
simple roots $\set{\alpha_1 ,\alpha_2 ,\ldots ,\alpha_{\theta}}$.
Obviously, the number of positive roots $p=|\Phi^+|$ is at least
$\theta$. The maps $g\colon \mathbf Z[I]\to G$ and $\chi\colon
\mathbf Z[I]\to\tilde G$ given by
$g_{\alpha}=g_1^{n_1}g_2^{n_2}\ldots g_{\theta}^{n_{\theta}}$ and
$\chi_{\alpha}=\chi^{n_1}\chi_2^{n_2}\ldots\chi_{\theta}^{n_{\theta}}$
for $\alpha =\sum_{i=1}^{\theta}n_i\alpha_i$, respectively, are
group homomorphisms. The bilinear map $q\colon \mathbf Z
[I]\times\mathbf Z [I]\to k^x$ defined by
$q_{\alpha_i\alpha_j}=q_{ij}$ can be expressed as
$q_{\alpha\beta}=\chi_{\beta}(g_{\alpha})$.

If $\cal X$ the set of connected components of the Dynkin diagram
of $\Phi$ let $\Phi_J$ be the root system of the component
$J\in\cal X$. The partition of the Dynkin diagram into connected
components corresponds to an equivalence relation on $I=\set{
1,2,\ldots ,\theta}$, where $i\sim j$ if $\alpha_i$ and $\alpha_j$
are in the same connected component.

\begin{Lemma} \cite[Lemma 2.3]{AS} Suppose that $\mathcal D$ is a
connected datum of finite Cartan type, i.e: the
Dynkin diagram of the Cartan matrix $(a_{ij})$ is connected, and such that
\begin{enumerate}
\item $q_{ii}$ has odd order, and
\item the order of $q_{ii}$ is prime to 3, if $(a_{ij})$ is of type $G_2$.
\end{enumerate}
Then there are integers $d_i\in\set{1,2,3}$ for $1\le i\le\theta$
and a $q\in k^x$ of odd order $N$ such that
$$q_{ii}=q^{2d_i} \quad . \quad d_ia_{ij}=d_ja_{ji}$$
for $1\le i,j\le\Theta$. If the Cartan matrix $(a_{ij})$ of
$\mathcal D$ is of type $G_2$ then the order of $q$ is prime to 3.
In particular, the $q_{ii}$ all have the same order in
$k^{\times}$, namely $N$.
\end{Lemma}

More generally, let $\mathcal D$ be a datum of finite Cartan type
in which the order $N_i$ of $q_{ii}$ is odd for all $i$, and the
order of $q_{ii}$ is prime to 3 for all $i$ in a connected
component of type $G_2$. It then follows that the order function
$N_i$ is constant, say equal to $N_J$, on each connected component
$J$. A datum satisfying these conditions will be called special
datum of finite Cartan type.

Fix a reduced decomposition of the longest element
$$w_0=s_{i_1}s_{i_2}\ldots s_{i_p}$$
of the Weyl group $W$ in terms of the simple reflections. Then
$$\set{s_{i_1}s_{i_2}\ldots s_{i_{l-1}}(\alpha_{i_l} )}_{i=1}^p$$
is a convex ordering of the positive roots.

Let $V=V(\mathcal D)$ be the crossed $kG$-module with basis $\set{
x_1,x_2,\ldots ,x_{\theta}}$, where $x_i\in V_{g_i}^{\chi_i}$ for
$1\le i\le\theta$. Then for all $1\le i\ne j\le\Theta$ the
elements $ad^{1-a_{ij}}x_i(x_j)$ are primitive in the free braided
Hopf algebra $\mathcal A(V)$ (see Lemma \ref{l17} or
\cite[Appendix 1]{AS1}). If $\mathcal D$ is as in the previous
Lemma then $\chi^{1-a_{ij}}\chi_j\ne\ep$. This implies that
$f(u_{ij})=0$ for any braided (Hopf) subalgebra $A$ of $\mathcal
A(V)$ containing $u_{ij}=ad^{1-a_{ij}}x_i(x_j)$ and any
$G$-invariant algebra map $f\colon A\to k$. Define root vectors in
$\mathcal A(V)$ as follows by iterated braided commutators of the
elements $x_1,x_2,\ldots ,x_{\theta}$, as in Lusztig's case but
with the general braiding:
$$x_{\beta_l}=T_{i_1}T_{i_2}\ldots T_{i_{l-1}}(x_{i_l}),$$
where $T_i(x_j)=\ad_{x_i}^{-a_{ij}}(x_j)$

In the quotient Hopf algebra $R(\mathcal D)=\mathcal
A(V)/(ad^{1-a_{ij}}x_i(x_j)|1\le i\ne j\le\theta )$ define root
vectors $x_{\alpha}\in\mathcal A(V)$ for $\alpha\in\Phi^+$ by the
same iterated braided commutators of the elements $x_1, x_2,\ldots
,x_{\theta}$ as in Lusztig's case but with respect to the general
braiding. (See \cite{AS2}, and the inductive definition of root
vectors in \cite{Ri} or also \cite[Section 8.1 and Appendix]{CP}.)
Let $K(\mathcal D)$ be the subalgebra of $R(\mathcal D)$ generated
by $\setst{ x_{\alpha}^N}{\alpha\in\Phi^+}$.

\begin{Theorem}\cite[Theorem 2.6]{AS}\label{ASth2.6} Let
$\mathcal D$ be a connected datum of finite Cartan type as in the
previous Lemma. Then
\begin{enumerate}
\item $\setst{x_{\beta_1}^{a_1}x_{\beta_2}^{a_2}\ldots x_{\beta_p}^{a_p}}
{a_1, a_2,\ldots , a_p\ge 0}$
forms a basis of $R(\mathcal D)$,
\item $K(\mathcal D)$ is a braided Hopf subalgebra of $R(\mathcal D)$ with basis
$$\setst{x_{\beta_1}^{Na_1}x_{\beta_2}^{Na_2}\ldots x_{\beta_p}^{Na_p}}{ a_1, a_2, \ldots , a_p\ge 0},$$
\item $[x_{\alpha},x_{\beta}^N]_c=0$, i.e: $x_{\alpha}x_{\beta}^N=
q_{\alpha\beta}^Nx_{\beta}^Nx_{\alpha}$ for all
$\alpha ,\beta \in \Phi^+$.
\end{enumerate}
\end{Theorem}

The vector space $V=V(\mathcal D)$ can also be viewed as a crossed
module in $^{\mathbf Z[I]}_{\mathbf Z[I]}YD$. The Hopf algebra
$\mathcal A(V)$, the quotient Hopf algebra $R(\mathcal D)=\mathcal
A(V)/(ad^{1-a_{ij}}x_i(x_j)|1\le i\ne j\le\theta )$ and its Hopf
subalgebra $K(\mathcal D)$ generated by $\setst{
x_{\alpha}^N}{\alpha\in\Phi^+}$ are all Hopf algebras in
$^{\mathbf Z[I]}_{\mathbf Z[I]}YD$. In particular, their
comultiplications are $\mathbf Z[I]$-graded. By construction, for
$\alpha\in\Phi^+$, the root vector $x_{\alpha} \in R(\mathcal D)$
is $\mathbf Z[I]$-homogeneous of $\mathbf Z[I]$-degree $\alpha$,
so that $x_{\alpha}\in R(\mathcal
D)_{g_{\alpha}}^{\chi_{\alpha}}$. To simplify notation write for
$1\le l\le p$ and for $\underline a =(a_1, a_2,\ldots ,
a_p)\in\mathbf N^p$
$$h_l=g_{\beta_l}^N\ ,\ \eta_l=\chi_{\beta_l}^N\ ,\ z_l=x_{\beta_l}^N$$
and $\underline a=\sum_{i=1}^pa_i\beta_i$
$$h^{\underline a}=h_1^{a_1}h_2^{a_2}\ldots h_p^{a_p}\in G\ ,\
\eta^{\underline
a}=\eta_1^{a_1}\eta_2^{a_2}\ldots\eta_p^{a_p}\in\tilde G\ , \
z^{\underline a}=z_1^{a_1}z_2^{a_2}\ldots z_p^{a_p}\in K(\mathcal
D).$$ In particular, for $e_l=(\delta_{kl})_{1\le k\le l}$, where
$\delta_{kl}$ is the Kronecker symbol, $e_l=\beta_l$ and
$z^{e_l}=z_l$ for $1\le l\le p$. The height of $\alpha
=\sum_{i=1}^{\theta}n_i\alpha_i\in\mathbf Z[I]$ is defined to be
the integer $ht(\alpha )=\sum_{i=1}^{\theta}n_i$. Observe that if
$\underline a,\underline b,\underline c\in\mathbf N^p$ and
$\underline a=\underline b+\underline c$ then
$$h^{\underline a}=h^{\underline b}h^{\underline c}\ ,\ \eta^{\underline a} =
\eta^{\underline b}\eta^{\underline c}\ \rm{and}\ ht(\underline
b)<ht(\underline a)\ \rm{if}\ \underline c\ne 0.$$ The
comultiplication on $K(\mathcal D)$ is $\mathbf Z[I]$-graded, so
that
$$\Delta_{K(\mathcal D)}(z^{\underline a})=z^{\underline a}\ot 1 + 1\ot z^{\underline a} +
\sum_{\underline b,\underline c\neq 0; \underline b+\underline
c=\underline a}t^{\underline a}_{\underline b\underline
c}z^{\underline b}\ot z^{\underline c}$$ and hence
$$\Delta_{K(\mathcal D)\#kG}(z^{\underline a})=z^{\underline a}\ot 1 + h^{\underline a}\ot z^{\underline a} +
\sum_{\underline b,\underline c\neq 0; \underline b+\underline
c=\underline a}t^{\underline a}_{\underline b\underline
c}z^{\underline b}h^{\underline c}\ot z^{\underline c}$$ on the
bosonization. The algebra $K(\mathcal D)$ is generated by the
subspace $L(\mathcal D)$ with basis $\set{z_1, z_2,\ldots , z_p}$.
The (left) $kG$-module structure on $\mathcal A(V)$ restricts to
$L(\mathcal D)$, and induces (right) $kG$-actions on $\Alg
(K(\mathcal D),k)$ and on $\Vect (L(\mathcal D),k)$ by the formula
$(fg)(x)=f(gx)$. A linear functional $f\colon L(\mathcal D)\to k$
is called $g$-invariant if $fg=f$ for all $g\in G$. Let
$\Vect_G(L(\mathcal D),k)$ be the subspace of $G$-invariant linear
functionals in $\Vect (L(\mathcal D),k)$.

\begin{Proposition}\label{connected} Let $\Vect_G(L(\mathcal D)$ and
$\Alg_G(K(\mathcal D),k)$ be the space of $G$-invariant linear
functionals and the set of $G$-invariant algebra maps, where
$\mathcal D$ is a connected special datum of finite Cartan type.
Then:
\begin{enumerate}
\item $\Vect_G(L(\mathcal D),k)= \setst{f\in\Vect (L(\mathcal
D),k)}{f(z_l)=0\ \rm{if}\ \eta_l\ne\ep}$. \item The restriction
map $\res\colon \Alg_G(K(\mathcal D),k)\to\Vect_G(L(\mathcal
D),k)$ is a bijection. The inverse is given by
$\res^{-1}(f)(z^{\underline a})=f(z_1)^{a_1})f(z_2^{a_2})\ldots
f(z_p^{a_p})$. \item $\Alg_G(K(\mathcal D),k)$ is a group under
convolution. \item The restriction map $\res\colon
{_G{\Alg_G}}(K(\mathcal D)\# kG,k)\to \Alg_G(K(\mathcal D),k)$ is
an isomorphism of groups with inverse defined by
$\res^{-1}(f)(x\ot g)=f(x)$, and ${_G\Alg_G}(K(\mathcal D)\#
kG,k)= \setst{\tilde f\in\Alg(K(\mathcal D)\#kG,k)}{\tilde
f_{|kG}=\ep}$. \item The map $\Theta =\rho\res^{-1}\colon
\Alg_G(K(\mathcal D),k)\to \Aut_{Hopf} (K(\mathcal D)\# kG)$,
defined by $\Theta (f) =\res^{-1}(f)*1*\res^{-1}(f)s$, is a group
homomorphism whose image is a subgroup in
$$\widetilde{\Aut}_{Hopf}(K(\mathcal D)\# k
G)=\setst{f\in\Aut_{Hopf}(K(\mathcal D)\# k G)}{f_{|kG}=id}.$$
\item For every $f\in\Alg_G(K(\mathcal D),k)$ the automorphism
$\Theta (f)$ of $K(\mathcal D)\#kG$ is determined by
\begin{eqnarray*}
&&\Theta (f)z^{\underline a} = z^{\underline a} +f(z^{\underline
a})(1-h^{\underline a})+ \sum_{\underline b,\underline c\ne
0;\underline b+\underline c=\underline a}t^{\underline
a}_{\underline b\underline c}
f(z^{\underline b})z^{\underline c} \\
&& +\sum_{\underline b,\underline c\ne 0;\underline b+\underline
c=\underline a}t^{\underline a}_{\underline b\underline
c}\left[z^{\underline b}+f(z^{\underline b})(1-h^{\underline b})
 +\sum_{\underline d,\underline e\ne 0;\underline d+\underline e=\underline b}
 t^{\underline b}_{\underline d,\underline e}f(z^{\underline d})z^{\underline e}\right]h^{\underline c} fs(z^{\underline c}).
\end{eqnarray*}
In particular, $\Theta (z^{\underline a})=z^{\underline a}
+f(z^{\underline a})(1-h^{\underline a})$ if $ht(\underline a)=1$.
\end{enumerate}
\end{Proposition}

\begin{proof} If $f\in\Vect_G(L(\mathcal D),k)$ then
$f(z_i)=f(gz_i)=\eta_i(g)f(z_i)$ for all
$1\le i\le p$ and for all $g\in G$.
Thus, $f(z_i)=0$ if $\eta_i\ne\ep$.

By Theorem \ref{ASth2.6} it follows that $K(\mathcal D)\cong
TL(\mathcal D)/(z_iz_j-\eta_j(h_i)z_jz_i|1\le i<j\le p)$. If
$f\in\Vect_G(L(\mathcal D),k)$ then the induced algebra map
$\tilde f\colon TL(\mathcal D)\to k$ factors uniquely through
$K(\mathcal D)$, since $\tilde
f(z_iz_j-\eta_j(h_i)z_jz_i)=f(z_i)f(z_j)-
\eta_j(h_i)f(z_j)f(z_i)=f(z_i)(f(z_j)-f(h_iz_j))=0$
for $1\le i,j\le p$, by the fact that $f$ is $G$-invariant. This
proves the second assertion.

The next three assertions are a special case of \ref{invariant}.

The set of all algebra maps $\Alg (K(\mathcal D),k)$ may not be a
group under convolution, but the subset $\Alg_G(K(\mathcal D),k)$
is. If $f_1$, $f_2$ and $f$ are $G$-invariant then
\begin{eqnarray*}
f_1*f_2(x y) & = & (f_1\ot f_2)(m\ot m)(1\ot c\ot 1)(x_1\ot x_2\ot y_1\ot y_2) \\
& & =(f_1\ot f_2)(x_1(x_2)_{-1}y_1\ot (x_2)_0y_2 \\
& & =f_1(x_1)\ep ((x_2)_{-1})f_1(y_1)f_2((x_2)_0)f_2(y_2) \\
& & =f_1(x_1)f_1(y_1)f_2(x_2)f_2(y_2) =f_1*f_2(x)f_1*f_2(y)
\end{eqnarray*}
and moreover, $(f_1*f_2)g= f_1g*f_2g =f_1*f_2$, $fsg=fgs=fs$, $\ep
*f =f =f*\ep$ and $f*fs =\ep = fs*f$ so that $\Alg_G(K(\mathcal
D)$ is closed under convolution multiplication and inversion.

The map $\Psi\colon \Alg_G(k(\mathcal D),k)\to \Alg (K(\mathcal
D)\# kG,k)$ given by $\Psi (f)(x\ot g)=f(x)$, is a homomorphism,
since
\begin{eqnarray*}
\res^{-1}(f_1)*\res^{-1} (f_2)(x\ot g)\hskip -2pt &=&
\res^{-1}(f_1)\ot\res^{-1} (f_2)(x_1\ot (x_2)_{-1}g\ot (x_2)_0\ot g) \\
& = & f_1(x_1)\ep ((x_2)_{-1})f_2((x_2)_0) \\
& = & f_1(x_1)f_2(x_2) = f_1*f_2(x) \\
& = & \res^{-1} (f_1)*\res^{-1} (f_2)(x\ot g).
\end{eqnarray*}
The inverse $\Psi^{-1}\colon \widetilde{\Alg}(K(\mathcal D)\#
kG,k)\to \Alg_G(K(\mathcal D),k)$ is given by $\Psi^{-1}(\tilde
f)(x)=\tilde f(x\ot 1)$, is just the restriction map.

It is convenient to use the notation $\Psi (f)=\tilde f$. Then
$\Theta (f) = \tilde f*1*\tilde fs$ and
\begin{eqnarray*}
\Theta (f_1*f_2) & = & \widetilde{f_1*f_2}*1*\widetilde{f_1*f_2}s \\
& = & \tilde f_1*\tilde f_2*1*\tilde f_2s*\tilde f_1s \\
& = & \Theta (f_1)\Theta (f_2).
\end{eqnarray*}
In particular,
$\Theta (f)\Theta (fs) =\Theta (f*fs)=
\Theta (\ep )=1=\Theta (fs*f)=\Theta (fs)\Theta (f)$.
Moreover,
\begin{eqnarray*}
\Theta (f) (x y) & = & \tilde f(x_1y_1)x_2y_2\tilde fs(x_3y_3) \\
& = & \tilde f(x_1)\tilde f(y_1)x_2y_2\tilde fs(y_3)\tilde f(x_3) \\
& = & \tilde f(x_1)x_2\tilde fs(x_3)\tilde f(y_1)y_2\tilde fs(y_3) \\
& =  & \Theta (f)(x)\Theta (f)(y)
\end{eqnarray*}
and
\begin{eqnarray*}
\Delta\Theta (f) & = & \Delta (\tilde f*1*\tilde fs) \\
& = & \Delta (\tilde f\ot 1\ot\tilde fs)\Delta^{(2)} \\
& = & (\tilde f\ot 1\ot 1\tilde fs)\Delta^{(3)} \\
& = & (\tilde f\ot 1\ot\ep\ot 1\ot\tilde fs)\Delta^{(4)} \\
& = & (\tilde f\ot 1\ot\tilde fs\ot\tilde f\ot 1\tilde fs)\Delta^{(5)} \\
& = & (\tilde f*1*\tilde fs\ot\tilde f*1*\tilde fs)\Delta \\
& = & (\Theta (f)\ot\Theta (f))\Delta ,
\end{eqnarray*}
showing that $\Theta (f)$ is an automorphism of $K(\mathcal D)\#
kG$ with inverse $\Theta (fs)$.

The remaining item now follows from the formula for the
comultiplication
$$\Delta (z^{\underline a}) = z^{\underline a}\ot 1 + h^{\underline a}\ot z^{\underline a} +
\sum_{\underline b,\underline c\ne 0;\underline b+\underline
c=\underline a}t^{\underline a}_{\underline b\underline
c}z^{\underline b}h^{\underline c}\ot z^{\underline c}$$ of
$K(\mathcal D)\# kG$, which implies
\begin{eqnarray*}
&&\hskip -10pt \Delta^{(2)}(z^{\underline a}) = z^{\underline
a}\ot 1\ot 1 +h^{\underline a}\ot z^{\underline a}\ot 1
+ \sum t^{\underline a}_{\underline b\underline c}z^{\underline b}h^{\underline c}\ot z^{\underline c}\ot 1 +
h^{\underline a}\ot h^{\underline a}\ot z^{\underline a} \\
& & + \sum t^{\underline a}_{\underline b\underline
c}\left[z^{\underline b}h^{\underline c}\ot h^{\underline c}\ot
z^{\underline c} + h^{\underline a}\ot z^{\underline
b}h^{\underline c}\ot z^{\underline c} +\sum t^{\underline
b}_{\underline r\underline l}z^{\underline r}h^{\underline
l+\underline c}\ot z^{\underline l}h^{\underline c}\ot
z^{\underline c}\right].
\end{eqnarray*}
and
$$1*s(z^{\underline a})=z^{\underline a} + h^{\underline a}s(z^{\underline a}) +
\sum t^{\underline a}_{\underline b\underline c}z^{\underline
b}h^{\underline c}s(z^{\underline c}) =\ep (z^{\underline a})=0.$$
Applying $\tilde f$ to the latter gives
$$f(z^{\underline a})+fs(z^{\underline a}) + \sum t^{\underline a}_{\underline b\underline c}f(z^{\underline b})fs(z^{\underline c}) = 0,$$
which will be used in the following evaluation. Now compute
\begin{eqnarray*}
\Theta (f)(z^{\underline a}) & = & f(z^{\underline a}) + z^{\underline a} +
\sum t^{\underline a}_{\underline b\underline c}
f(z^{\underline b})z^{\underline c} + h^{\underline a}fs(z^{\underline a}) \\
& & + \sum t^{\underline a}_{\underline b\underline
c}[f(z^{\underline b}) + z^{\underline b} +
\sum t^{\underline b}_{\underline r\underline l}f(z^{\underline r})z^{\underline l}]h^{\underline c}fs(z^{\underline c}) \\
& = & z^{\underline a} +f(z^{\underline a})(1-h^{\underline a}) + \sum t^{\underline a}_{\underline b\underline c}f(z^{\underline b})
z^{\underline c} \\
& & + \sum t^{\underline a}_{\underline b\underline
c}[z^{\underline b} + f(z^{\underline b})(1-h^{\underline b}) +
\sum t^{\underline b}_{\underline r\underline l}f(z^{\underline
r})z^{\underline l}]h^{\underline c}fs(z^{\underline c})
\end{eqnarray*}
to get the required result.
\end{proof}

For any $f\in\Alg_G(K(\mathcal D),k)$ define by induction on
$ht(a)$ the following elements in the augmentation ideal of $kG$
$$u_{\underline a}(f)= f(z^{\underline a})(1-h^{\underline a})+\sum_{\underline b,\underline c\ne 0;\underline b+\underline c=\underline a}
t^{\underline a}_{\underline b\underline c}f(z^{\underline
b})u_{\underline c}(f),$$ where $u_{\underline
a}(f)=f(z^{\underline a})(1-h^{\underline a})$ if $ht (\underline
a)=1$. In particular, for a positive root $\alpha =\beta_l\in
\Phi^+$ and $x_{\alpha}^N=x_{\beta_l}^N=z^{e_l}$ write
$u_l(f)=u_{e_l}(f)=u_{\alpha}(f)$. We can think of
$f=(f(x_{\alpha}^N)|\alpha\in\Phi^+)$ as root vector parameters in
the sense of \cite{AS}.

\begin{Corollary} Let $\mathcal D$ be a special connected
datum of Cartan type. Then
$$U(\mathcal D, f) = R(\mathcal D)\#kG/(x_{\alpha}^N+u_{\alpha}(f))$$
are the liftings of $\mathcal B(V)\#kG =U(\mathcal D, \ep )$.
\end{Corollary}

\begin{proof} The augmentation ideal of $K(\mathcal D)$,
the ideal $I$ of $K(\mathcal D)\#kG$ and the ideal $(I)$ in
$R(\mathcal D)\#kG$ generated by $\setst{x_{\alpha}}{\alpha\in
\mathcal X}$ are Hopf ideals. It follows from the inductive
formulas for $\Theta (f)(z^{\underline a})$ and $u_{\underline
a}(f)$ above that for every $f\in\Alg_G(K(\mathcal D),k)$ the
ideals $I_f=\Theta (f)(I)$ in $K(\mathcal D)\#kG$ and $(I_f)$ in
$R(\mathcal D)\#kG$ generated by
$\setst{x_{\alpha}^N+u_{\alpha}(f)}{\alpha\in \Phi^+}$ are Hopf
ideals as well. The Hopf algebras $U(\mathcal D,f)=K(\mathcal
D)\#kG/(\Theta (f)(I))$ are the liftings of $U(\mathcal D,\ep
)=\mathcal B(V)\# kG$ parameterized by
$f=(f(x_{\alpha}^N|\alpha\in\Phi^+)\in\Alg_G(K(\mathcal D),k).$
\end{proof}

In the not necessarily connected case of a special datum of finite
Cartan type the elements $ad^{1-a_{ij}}x_i(x_j)$ are still
primitive in $\mathcal A(V)$ and $R(\mathcal D)=\mathcal
A(V)/(ad^{1-a{ij}}x_i(x_j)|i\sim j)$ is still a Hopf algebra,
which contains $R(\mathcal D_J)$ for every connected component
$J\in\mathcal X$. The Hopf subalgebra $K(\mathcal D)$ generated by
the subspace with basis  $S=\setst{z_J^{a_J}, z_{ij}}{J\in\mathcal
X, i\not\sim j}$, where $z_{ij}=[x_i,x_j]_c$, contains $K(\mathcal
D_J)$ for every $J\in\mathcal X$. The comultiplication in each
components $K(\mathcal D_J)$ and $K(\mathcal D_J)\#kG$ is of
course given as before in the connected case, while for $i\not\sim j$
$$\Delta (z_{ij})=z_{ij}\ot 1+1\ot z_{ij}$$
in $K(\mathcal D)$ and $R(\mathcal D)$ and
$$\Delta (z_{ij})=z_{ij}\ot 1 + g_ig_j\ot z_{ij}$$
in the bozonizations $K(\mathcal D)\#kG$ and $R(\mathcal D)\#kG$.
The space of $G$-invariant linear functionals $\Vect_G(L(\mathcal
D),k)$ consists elements $f\in\Vect (L(\mathcal D),k)$ such that
$$f(z_r)=0\ \rm{if}\ \eta_r\ne\ep\ \rm{for}\ 1\le r\le p\ \rm{and}\ f(z_{ij})=0\ \rm{if}\
\chi_i\chi_j\ne\ep\ \rm{if}\ i\not\sim j \}.$$
The induced algebra map $\tilde f\colon TL(\mathcal D)\to k$ of
such a linear functional satisfies
\begin{itemize}
\item $\tilde f([z_r,z_s]_c)= f(z_r)(f(z_s)-f(h_rz_s))=0,$
\item $\tilde f([z_{ij},z_r]_c)=f(z_{ij})(f(z_r)-f(g_ig_jz_r)=0,$
\item $\tilde f([z_{ij},z_{lm}]_c)=f(z_{ij})(f(z_{lm})-f(g_ig_jz_{lm}))=0,$
\end{itemize}
since $f$ is $G$-invariant. It therefore factors through
$K(\mathcal D)$, since
$$TL(\mathcal D)/([z_r,z_s]_c, [z_{ij},z_r]_c, [z_{ij},z_{lm}]_c)
=K(\mathcal D)/([z_r,z_s]_c, [z_{ij},z_r]_c, [z_{ij},z_{lm}]_c).$$
It follows that the restriction maps
$$\res \colon  _G\Alg_G(K(\mathcal D)\#kG)\to \Alg_G(K(\mathcal D),k)\to
\Vect_G(L(\mathcal D),k)$$
are bijective, and $f=\setst{
f(z_{ij})}{i\not\sim j}\cup\setst{f(z_r)}{1\le r\le p}$ can be
interpreted as a combination of linking parameters and root vector
parameters in the sense of \cite{AS}. Then map
$$\Theta :\Alg_G(K(\mathcal D),k)\to \Aut_{Hopf}(K(\mathcal D)\#kG)$$
given by $\Theta (f)=f*1*fs$ is a homomorphism of groups.
Moreover, since $z_{ij}+g_ig_js(z_{ij})=m(1\ot s)\Delta
(z_{ij})=0$ in $K(\mathcal D)\#kG$, it follows that
$$\Theta (f)(z_{ij})=
(f\ot 1\ot fs)\Delta^{(2)}(z_{ij})=z_{ij}+f(z_{ij})(1-g_jg_j)$$
when $i\not\sim j$, while $\Theta f(z_r)$ is given inductively as
in \ref{connected}. In this way one obtains therefore all the
`liftings' of $B(V)\# kG$ for special data of finite Cartan type.

\begin{Theorem}\label{main3} Let $\mathcal D$ be a special datum of
finite Cartan type. Then
$$U(\mathcal D,f)= R(\mathcal D)\#kG/(x_{\alpha}^{N_{\alpha}}+u_{\alpha}(f),
[x_i,x_j]_c+f(z_{ij})(1-g_ig_j)|\alpha\in\Phi^+ , i\not\sim j)$$
for $f\in\Vect_G(L(\mathcal D),k)$ are the liftings of $\mathcal
B(V)\#kG =U(\mathcal D, \ep )$. Moreover, all these liftings are
monoidally Morita-Takeuchi equivalent.
\end{Theorem}

\begin{proof} Clearly, $U(\mathcal D,f)$ is a lifting of
$\mathcal B(V)\#kG$ for the root vector parameters $\setst{
\mu_{\alpha}=f(x_{\alpha}^{N_{\alpha}}}{\alpha\in\Phi^+}$ and the
linking parameters $\setst{\lambda_{ij}=f([x_i,x_j]_c)}{i\not\sim
j}$. By \cite{AS} all liftings of $\mathcal B(V)\#kG$ are of that
form. To proof the last assertion let in \ref{po} $H=kG$,
$K=K{\mathcal D}$, $f\in\Alg_G(K,k)$. Then the ideal
$I=(x_{\alpha}^{n_\alpha}, [x_i,x_j]_c|\alpha\in\Phi^+, i\not\sim
j)$  and $J=\Theta (f)(I)=(x_{\alpha}^{N_{\alpha}}+u_{\alpha}(f),
[x_i,x_j]_c+f(z_{ij})(1-g_ig_j)|\alpha\in\Phi^+ , i\not\sim j)$ of
$K\#kG$ are conjugate. By \ref{po} the quotient Hopf algebras
$U(\mathcal D,\ep)$ and $U(\mathcal D, f)$ of $R(\mathcal D)\# kG$ are monoidally
Morita-Takeuchi equivalent. The additional condition $(R\# kG)/(\res^{-1}(f)*(I\# kG))\ne 0$
is verified in (\cite{Ma2}, Appendix).
\end{proof}

\section{Cocycle deformations and cohomology}

In this section we describe liftings of special crossed modules
$V$ over finite abelian groups in terms of cocycle deformations of
$B(V)\#kG$, and determine the infinitesimal part of the
deformations by means of Hochschild cohomology.

\subsection{Cocycle deformations} A normalized 2-cocycle
$\sigma\colon A\ot A\to k$ on a Hopf algebra $A$ is a convolution invertible linear map such that
$$(\ep\ot\sigma )*\sigma (1\ot m)=(\sigma\ot\ep )*\sigma (m\ot 1)$$
and $\sigma (\iota\ot 1)=\ep =\sigma (1\ot\iota )$. The deformed
multiplication
$$m_{\sigma}=\sigma *m*\sigma^{-1}\colon A\ot A\to A$$
and antipode
$$s_{\sigma}=\sigma*s*\sigma^{-1}\colon A\to A$$
on $A$, together with the original unit, counit and
comultiplication define a new Hopf algebra structure on H which we
denote by $A_{\sigma}$. If $A$ is $\mathbf N$-graded then $\sigma
=\sum_{i=0}^{\infty}\sigma_i$, where $\sigma_j\colon A\ot A\to k$
is the uniquely determined component of degree $-j$ and
$\sigma_0=\ep$. This corresponds to a convolution invertible
normalized 2-cocycle
$$\sigma (t)=\sum_{i=0}^{\infty}\sigma_it^i\colon A\ot A\to k[[t]].$$
The convolution inverse
$\sigma^{-1}(t)=\sum_{i=0}^{\infty}\eta_it^i\colon A\ot A\to
k[[t]]$ is determined by $\sigma (t)*\sigma^{-1}(t)=\ep
=\sigma^{-1}(t)*\sigma (t)$, that is by
$$\sum_{i+j=l}\sigma_i*\eta_j =\delta^l_0=
\sum_{i+j=l}\eta_i*\sigma_j.$$ The cocycle condition
$(\ep\ot\sigma(t))*\sigma(t)(1\ot m)=(\sigma(t)\ot\ep
)*\sigma(t)(m\ot 1)$ implies that
$$\sum_{i+j=l}(\ep\ot\sigma_i )*\sigma_j(1\ot m)=
\sum_{i+j=l}(\sigma_i\ot\ep )*\sigma_j(m\ot 1)$$
for all $l\ge 0$. In particular, if $s$ is the least positive
integer for which $\sigma_s\ne 0$ then $\eta_s=-\sigma_s$ and
$$\ep\ot\sigma_s + \sigma_s (1\ot m)=\sigma_s\ot\ep +\sigma_s(m\ot 1)$$
so that $\sigma_s\colon A\ot A\to k$ is a Hochschild 2-cocycle.
The infinitesimal part $\mod (t^{s+1})$ of $\sigma (t)$ and of
$m_{\sigma (t)}$ are
$$\ep\ot\ep + \sigma_st^s \colon  A\ot A\to k[t]/(t^{s+1})$$
and
$$m_{\sigma (t)}=m+(\sigma_s*m-m*\sigma_s)t^s\colon A\ot A\to A[t]/(t^{s+1}),$$
respectively, where
$$\phi =\sigma_s*m-m*\sigma_s\colon A\ot A\to A$$
is a normalized Hochschild 2-cocycle.

Dually, a normalized 2-cocycle $\sigma\colon k\to A\ot A$ is a
convolution invertible linear map such that
$$(\iota\ot\sigma )*(1\ot\Delta)\sigma =(\sigma\ot\iota )*(\Delta\ot 1)\sigma $$
and
$$(\ep\ot 1)\sigma =\iota = (1\ot\ep )\sigma .$$
Then $H^{\sigma} =(H, m, \iota ,\Delta^{\sigma} ,\ep )$ with the
deformed comultiplication
$$\Delta^{\sigma}=\sigma *\Delta *\sigma^{-1}\colon A\to A\ot A$$
is again a Hopf algebra. If $A$ is $(-\mathbf N)$-graded then
$\sigma =\sum_{i=0}^{\infty}\sigma_i$, where $\sigma_j\colon A\ot
A\to k$ is the uniquely determined component of degree $-j$,
corresponds to an invertible normalized 2-cocycle $\sigma
(t)=\sum_{i=0}^{\infty}\sigma_it^i\colon k\to A\ot A[[t]]$.

\begin{Theorem} \cite[Corollary 5.9]{Sch} If two Hopf algebras
$A$ and $A'$ are cocycle deformations of each other,
then they are monoidally Morita-Takeuchi equivalent.
The converse is true if $A$ and $A'$ are finite dimensional.
\end{Theorem}

Suppose now that $V$ is a crossed $kG$-module of special finite
Cartan type, $\cal A(V)$ the free braided algebra and $\cal A(V)\#
kG$ its bosonization. If $I$ is the ideal of $\cal A(V)$ generated
by the subset
$$S=\setst{ad^{1-a_{ij}}x_i(x_j)}{i\sim j}\cup
\setst{x_{\alpha}^{N_{\alpha}}}{\alpha\in\Phi^+}\cup\setst{
[x_i,x_j]_c}{i\not\sim j}$$ then $\cal A(V)/I=\cal B(V)$ is the
Nichols algebra. The subalgebra $K$ of $\cal A(V)$ generated by
$S$ is a Hopf subalgebra \cite{AS}, \cite[Proposition 9.2.1]{CP}.
Then $K\# kG$ is the Hopf subalgebra of $\cal A(V)\# kG$ generated
by $S$ and $G$.

\begin{Lemma} The injective group homomorphism
$$\phi \colon \Alg_G(K,k)\to \Alg (K\#kG,k)$$ given by $\phi (f)(x\#
g)=f(x)$ has image $$\widetilde{\Alg (K\#kG,k)}=\setst{f\in
\Alg(K\#kG,k)}{f_{|kG}=\ep}$$ and
$$\adj\colon \Alg_G(K,k)\to \Aut(K\#kG)$$
has its image in the subgroup $$\widetilde{Aut(K\#kG)}=\setst{
f\in\Aut (K\#kG)}{f_{|kG}=\ep }.$$ Moreover, if $V$ is of special
finite Cartan type then $f(ad^{1-a_{ij}}x_i(x_j))=0$ for $i\sim j$
and for every $f\in\Alg_G(K,k)$.
\end{Lemma}

\begin{proof} If $f\in \Alg (K\#kG,k)$ then
$f(ad^{1-a_{ij}}x_i(x_j))= f(g\cdot
ad^{1-a_{ij}}x_i(x_j)g^{-1})=\chi_i(g)^{1-a_{ij}}\chi_j(g)$,
$f(x_{\alpha}^N)=f(gx_{\alpha}^{N_{\alpha}}g^{-1})
=\chi_{\alpha}^{N_{\alpha}}(g)f(x_{\alpha}^N)$ and
$f([x_i,x_j]_c)=f(g[x_i,x_j]_cg^{-1})=\chi_i(g)\chi_j(g)f([x_i,x_j]_c)$,
so that $f(g\cdot ad^{1-a_{ij}}x_i(x_j))=0$ if
$\chi_i^{1-a_{ij}}\chi_j\ne 0$, $f(x_{\alpha}^{N_{\alpha}})=0$ if
$\chi_{\alpha}^{N_{\alpha}}\ne\ep $ and $f([x_i,x_j]_c)=0$ if
$\chi_i\chi_j\ne\ep$.
\end{proof}

The theorem above can now be applied to the situation in Section
\ref{s3} to show that all \lq liftings'  of a crossed $kG$-module
of special finite Cartan type are cocycle deformations of each
other. The special case of quantum linear spaces has been studied
by Masuoka \cite{Ma}, and that of a crossed $kG$-module
corresponding to a finite number of copies of type $A_n$ by Didt
\cite{Di}.

\begin{Theorem} Let $G$ be a finite abelian group, $V$ a crossed
$kG$-module of special finite Cartan type, $\mathcal B(V)$
its Nichols algebra with bosonization $A=\mathcal B(V)\#kG$. Then:
\begin{enumerate}
\item All liftings of $A$ are monoidally Morita-Takeuchi equivalent,
i.e: their comodule categories are monoidally equivalent, or equivalently,
\item all liftings of $A$ are cocycle deformations of each other.
\end{enumerate}
\end{Theorem}

\begin{proof} The main theorem \ref{main3} at the end the last section says that
$\mathcal B(V)\#kG\cong U(\mathcal D, \ep )\cong R(\mathcal
D)\#kG/(I)$ for a Hopf ideal $I$ in the Hopf subalgebra
$K(\mathcal D)\#kG$ of $R(\mathcal D)\#kG$, that its liftings are
of the form $U(\mathcal D,f)\cong R(\mathcal D)\#kG/(I_f)$ for a
conjugate Hopf ideal $I_f$, where $f\in\Alg_G(K(\mathcal D,k),
k)\cong \widetilde{\Alg}(K(\mathcal D)\#kG,k)$, and that they are
all Morita-Takeuchi equivalent. Thus, Schauenburg's result
applies, so that all these liftings are cocycle deformations of
each other.
\end{proof}

\begin{Corollary} Let $H$ be a finite dimensional pointed Hopf algebra
with abelian group of points $G(H)=G$ and assume that the order of $G$ has no
prime divisors $< 11$. Then:
\begin{itemize}
\item $H$ and $\gr_c(H)$ are Morita Takeuchi equivalent, or equivalently,
\item $H$ is a cocycle deformation of $\gr_c(H)$.
\end{itemize}
\end{Corollary}

\begin{proof} Under the present assumptions the Classification Theorem \cite{AS}
asserts that $\gr_c(H)\cong B(V)\# kG$ for a crossed $kG$-module
$V$ of special finite Cartan type, and hence the previous theorem
applies.
\end{proof}

In the case at hand $A=B(V)\#kG$ and the condition that
$\gr_cA^{\sigma}\cong A$ implies that the cocycle $\sigma\colon
A\ot A\to k$ is  $G$-invariant, since $m_{\sigma}(x\ot g)=m(x\ot
g)$ and $m_{\sigma}(x\ot g)=m(x\ot g)$ for all $g\in G$, so that
$$\sigma (x\ot g)=\ep (x) =\sigma (g\ot x)$$
for all $g\in G$. The cocycle conditions then imply that
$$\sigma(x\ot yg)=\sigma(x\ot y) \ ,\ \sigma(xg\ot y)=
\sigma(x\ot gy) \ ,\ \sigma(gx\ot y)=\sigma(x\ot y),$$ which means
that $\sigma$ factors through $A\ot_{kG}A$ and also that the
cocycle really comes from a convolution invertible $G$-invariant
2-cocycle
$$\nu\colon B(V)\ot B(V)\to k.$$
In fact, the restriction of a $G$-invariant 2-cocycle $\sigma
\colon A\ot A\to k$ restricts to a $G$ invariant 2-cocycle on
$B(V)\ot B(V)$ and the map
$$\Psi\colon Z_G^2(B(V),k)\to Z_G^2(A,k),$$
defined by $\Psi (\nu )(x\# g\ot x'\# g')=\nu (x\ot g(x'))$, is
inverse to the restriction map. This map is of degree zero and
therefore also defines a bijection between the associated sets of
formal cocycles
$$\Psi\colon Z_G^2(B(V),k[[t]])\to Z_G^2(A,k[[t]]),$$
and the infinitesimal parts, which are Hochschild cocycles.

\subsection{Exponential map} It is in general very hard to
give explicit examples of multiplicative cocycles. One somewhat
accessible family consists of bicharacters. Below we give another
idea which can sometimes be used.

Note that if $B=\oplus_{n=0}^\infty B_n$ is a graded bialgebra, and
$f\colon B\to k$ is a linear map such that $f|_{B_0}=0$, then
$$
e^f=\sum_{i=0}^\infty \frac{f^{* i}}{i!}\colon B\to k
$$
is a well defined convolution invertible map with convolution
inverse $e^{-f}$. When $f\colon B\otimes B\to k$ is a Hochschild
cocycle such that $f|_{B\ot B_0+B_0\ot B}=0$, then `often'
$e^f\colon B\ot B\to k$ will be a multiplicative cocycle. For
instance this happens whenever $f(1\ot m)$ and $f(m\ot 1)$ commute
(with respect to the convolution product) with $\ep\ot f$ and
$f\ot \ep$, respectively. Also note that if $f*f=0$, then
$e^f=\ep+f$.

From now on assume $B$ is obtained as a bosonization of a quantum
linear space. More precisely $B=\genst{G, x_1,\ldots ,
x_\theta}{gx_i =\chi_i(g) x_i g, x_ix_j = \chi_j(g_i) x_j x_i,
x_i^{N_i}=0}$. Here $\chi_1,\ldots ,
\chi_\theta\in\widehat G$, $g_1,\ldots , g_\theta\in\Gamma$
are such that $\chi_i(g_j)\chi_j(g_i)=1$ for $i\not=j$. Number
$N_i$ is the order of $\chi_i(g_i)$. We abbreviate
$q_{i,j}=\chi_i(g_j)$. Then $\zeta_i\colon B\ot B\to k$, given by
$$\zeta_i(xg, yh)=\begin{cases}\chi_i^{b_i}(g) ,& \mbox{ if }
x=x_i^{a_i}, y=x_i^{b_i}, a_i+b_i=N_i \\
0 ,& \mbox{ otherwise }\end{cases}$$ for $x=x_1^{a_1}\ldots
x_\theta^{a_\theta}$ and $y=x_1^{b_1}\ldots x_\theta^{b_\theta}$
(see Corollary \ref{c46}) is a Hochschild cocycle. Moreover, each
of the sets $$A_l=\setst{(\ep\ot\zeta_i), \zeta_i(1\ot m)}{1\le
i\le\theta}$$ and
$$A_r=\setst{(\zeta_i\ot \ep), \zeta_i(m\ot 1)}{1\le
i\le\theta}$$ is a commutative set (for the convolution product). We
sketch the proof for $A_l$ (the proof for $A_r$ is symmetric).
Maps $\zeta_i(1\ot m)$ and $\zeta_j(1\ot m)$ commute since
$\zeta_i$ and $\zeta_j$ do. Same goes for $\ep\ot\zeta_i$ and
$\ep\ot\zeta_j$. Hence it is sufficient to prove that for all
$i,j$ we have
$$
(\ep\ot\zeta_i)*(\zeta_j(1\ot m))=  (\zeta_j(1\ot
m))*(\ep\ot\zeta_i).
$$
If $i\not=j$, this is immediate. For $i=j$ note that both left and
right hand side can be nonzero only at PBW elements of the form
$x_i^r f\ot x_i^s g\ot x_i^p h \in B\ot B\ot B$, with
$r+s+p=2N_i$. Without loss of generality assume that $f=g=h=1$. In
this case the left hand side evaluates to
\begin{eqnarray*}
\sum_{u+v=N_i} {s\choose u}_{q_{ii}} {p\choose v}_{q_{ii}}
q_{ii}^{u(p-v)} = 1
\end{eqnarray*}
and the right hand side is
\begin{eqnarray*}
\sum_{u+v=N_i-r} {s\choose u}_{q_{ii}} {p\choose v}_{q_{ii}}
q_{ii}^{u(p-v)} = 1.
\end{eqnarray*}
Thus if $f$ is any map in the linear span $Span_k \set{\zeta_i}$,
then $e^f$ is a multiplicative cocycle.

This idea is illustrated in some of the examples given in Section
\ref{ex1}.

\subsection{The standard cosimplicial algebra complex and
cohomology of braided Hopf algebras} The `multiplicative' cocycles
above and the `additive' Hochschild cocycles can in principle be
computed from the normalized standard cosimplicial complex
associated with the standard comonad $A\ot -$ on the category of
$A$-bimodules. The relevant part of that complex with coefficients
in the $A$-bimodule $M$ is
$$\Hom(k,M)^{\del_0\atop{\longrightarrow}}_{\del_1
\atop{\longrightarrow}}
\Hom(A,M)^{{\del_0\atop{\longrightarrow}}\atop{\del_1
\atop{\longrightarrow}}}_{\del_2\atop{\longrightarrow}}
\Hom(A^2,M)^{{\del_0\atop{\longrightarrow}}\atop{\del_1
\atop{\longrightarrow}}}_{{\del_2\atop{\longrightarrow}}
\atop{\del_3\atop{\longrightarrow}}}
\Hom(A^3,M)$$
with coface maps $\del_i\colon \Hom (A^n,M)\to\Hom
(A^{n+1},M)$ given by
$$\del_i(f) = \begin{cases}  \mu_l(1\ot f) & \rm{, if}\ i=0 \\
f(1^{i-1}\ot m\ot 1^{n-i-1}) & \rm{, if}\ 1\le i\le n-1 \\
(f\ot 1)\mu_r & \rm{, if}\ i=n
\end{cases}$$
and codegeneracy maps $s_i\colon \Hom (A^{n+1},M)\to\Hom (A^n,M)$,
$s_if=f(1^i\ot\iota\ot 1^{n-i})$, where $\iota\colon k\to A$ is
the unit. Hochschild (or the 'additive') cohomology $H^*(R,M)$ is
the cohomology of the associated cochain complex with the
alternating sum differentials $\del
=\sum_{i=0}^n(-1)^i\del_i\colon \Hom (A^n,M)\to \Hom (A^{n+1},M)$,
so that
\begin{eqnarray*}
\partial f(a_1\ot\ldots\ot a_{n+1}) &=& a_1f(a_2\ot\ldots\ot a_{n+1}) \\
&+&\sum_{i+1}^n(-1)^if(a_1\ot\ldots\ot a_ia_{i+1}\ot\ldots\ot a_{n+1}) \\
&+&(-1)^{n+1}f(a_1\ot\ldots\ot a_n)a_{n+1}.
\end{eqnarray*}

If $M=k$, the trivial $A$-bimodule, then the cosimplicial complex
is a cosimplicial algebra under convolution. Apply the group of
units functor to this cosimplicial algebra to get a generally
non-abelian cosimplicial group. Then
$$Z^1(A,k)=\setst{f\in\Hom (A,k)}{\del_2(f)*\del_0(f)=\del_1(f)} =\Alg (A,k)$$
is the group of `multiplicative' 1-cocycles, while
\begin{eqnarray*}
Z^2(A,k) & = & \setst{f\in\Hom (A\ot A,k)}{\del_3(f)*\del_1(f)=\del_0(f)*\del_2(f)} \\
 & = & \setst{ f\in\Hom (A^2,k)}{f(x_1,y_1)f(x_2y_2,z)=f(y_1,z_1)f(x,y_2z_2)}
\end{eqnarray*}
is the set of `multiplicative' 2-cocycles. In case $A$ is
cocommutative, the cosimplicial group is abelian and from the
associated cochain complex with the alternating convolution
product differentials one gets Sweedler cohomology.

This theory also works for a braided algebra in the category of
crossed $H$-modules when the tensor products are taken in the
braided sense.

\begin{Proposition} If $A$ and $A'$ are finite dimensional (braided) algebras then
$$H^*(A,M)\ot H^*(A',M')\cong H^*(A\ot A',M\ot M')$$
\end{Proposition}

\begin{proof} The Bar-Resolution $B(A,M)$ of the $A$-bimodule $M$,
with differential
$$d\colon B_{n+1}(A,M)=A\ot A^n\ot M\to A\ot A^{n-1}\ot M= B_n(A,M)$$
given by
$$d(a_0\ot a_1\ot\ldots\ot a_n\ot m)=\sum_{i=0}^n(-1)^ia_0\ot
\ldots\ot a_{i-1}\ot a_ia_{i+1}\ot a_{i+2}\ot\ldots\ot a_n\ot m,$$
is a $k$-split relatively free resolution of $M$. The Hochschild
cohomology of $A$ with coefficients in $M$ is defined by
$H^*(A,M)=H^*({_A\Hom_A}(B(A,A),M))$. Both, $B(A,M)\ot B(A',M')$
and $B(A\ot A',M\ot M')$ are $k$-split relatively free resolutions
of the $A\ot A'$-bimodule $M\ot M'$. By the comparison theorem the
two chain complexes are  chain equivalent. Such a chain
equivalence is given by the (braided) version of the
Alexander-Whitney map. By the (braided) version of the Kunneth
Theorem there is a natural map
$$H^*B(A,M)\ot H^*B(A',M')\to H^*(B(A,M)\ot B(A',M'))\cong H^*B(A\ot A',M\ot M')$$
which is an isomorphism when either $A$ or $A'$ is finite dimensional.
\end{proof}

This result can be applied to Nichols algebras of certain finite
dimensional Yetter-Drinfel'd modules over abelian groups.

\begin{Theorem} If $V=\oplus_{J\in\mathcal X}V_J$ is the crossed
$kG$-module of a special datum $\mathcal D$ of finite Cartan type,
where $\mathcal X$is the set of connected components of the Dynkin
diagram, then
$$B(V)\cong\ot_{J\in\mathcal X}B(V_J)$$
as a braided Hopf algebra and
$$H^*(B(V),k)\cong \ot_{J\in\mathcal X}H^*(B(V_J),k)$$
as a graded vector space.
\end{Theorem}\qed

\begin{Corollary}If $V=\oplus_{i=1}^tkx_i$ is a quantum linear space over
an abelian group $G$ then $B(V)\cong B_1\ot B_2\ot\ldots\ot B_t$,
where $B_i=B(kx_i)\cong k[x_i]/(x_i^{n_i})$. Moreover,
$$H^*(B(V))\cong H^*(B_1)\ot H^*(B_2)\ot\ldots\ot H^*(B_t)$$ with
$H^j(B_i,B_i)\cong k[x_i]/(x_i^{n_i-1})$ and $H^j(B_i,k)\cong k$.
\end{Corollary}\qed

\begin{Remark}\label{c46} Note that if $A=B(kx)=k[x]/(x^n)$, then
\begin{eqnarray*}
H^0(A,k)=k \mbox{ and } H^1(A,k)=\Vect (A^+/(A^+)^2,k).
\end{eqnarray*}
A normalized 2-cocycle $f\colon A^+\ot A^+\to k$ is a linear map
satisfying $f(x^i\ot x^j)=f(x^k\ot x^l)$ whenever $i+j=k+l$, so
that $Z^2(A,k)=\oplus kf_l$, where $f_l(x^i\ot x^j)=1$ if $i+j=l$
and $f_l(x^i\ot x^j)=0$ otherwise. If $f(x^i\ot x^j)=0$ for
$i+j=n$ then $f=\delta g$, where $g(x^{i+j})=f(x^i\ot x^j)$, so
that $H^2(A,k)=Z^2(A,k)/B^2(A,k)$ is represented by
$f_n$.
\end{Remark}

\subsection{The equivariant cohomology}
The $G$-invariant Hochschild cocycles are described via the
cosimplicial complex of $G$-invariant elements in the standard
complex. The commutative \lq pushout-pullback' square of (braided)
Hopf algebras in Section \ref{s2}
$$\begin{CD}
K @>\kappa >> R \\
@V\ep VV  @V\pi VV \\
k @>\iota >> B
\end{CD}$$
induces a square of cosimplicial algebras

$$\begin{array}{ccccccc}
\Hom_G(k,k) & ^{\del_0\atop{\longrightarrow}}_{\del_1\atop{\longrightarrow}} & \Hom_G (B,k) & ^{{\del_0\atop{\longrightarrow}}\atop{\del_1\atop{\longrightarrow}}}_{\del_2\atop{\longrightarrow}} & \Hom_G (B^2,k) & ^{{\del_0\atop{\longrightarrow}}\atop{\del_1\atop{\longrightarrow}}}_{{\del_2\atop{\longrightarrow}}\atop{\del_3\atop{\longrightarrow}}} & \Hom_G (B^3,M) \\
\|  & & \downarrow\pi^* & & \downarrow (\pi^2)^* & & \downarrow (\pi^3)^* \\
\Hom_G(k,k) & ^{\del_0\atop{\longrightarrow}}_{\del_1\atop{\longrightarrow}} &\Hom_G (R,k) & ^{{\del_0\atop{\longrightarrow}}\atop{\del_1\atop{\longrightarrow}}}_{\del_2\atop{\longrightarrow}} & \Hom_G (R^2,k) & ^{{\del_0\atop{\longrightarrow}}\atop{\del_1\atop{\longrightarrow}}}_{{\del_2\atop{\longrightarrow}}\atop{\del_3\atop{\longrightarrow}}} & \Hom_G (R^3,M) \\
\|  & & \downarrow\kappa^* & & \downarrow (\kappa^2)^* & & \downarrow (\kappa^3)^* \\
\Hom_G(k,k) & ^{\del_0\atop{\longrightarrow}}_{\del_1\atop{\longrightarrow}} & \Hom_G (K,k) & ^{{\del_0\atop{\longrightarrow}}\atop{\del_1\atop{\longrightarrow}}}_{\del_2\atop{\longrightarrow}} & \Hom_G (K^2,k) & ^{{\del_0\atop{\longrightarrow}}\atop{\del_1\atop{\longrightarrow}}}_{{\del_2\atop{\longrightarrow}}\atop{\del_3\atop{\longrightarrow}}} & \Hom_G (A^3,M)
\end{array}$$
where the trivial part has been omitted. Here is a equivariant analog of the 5-term sequence, which allows a direct calculation of the infinitesimal deformation cocycle associated with the equivariant algebra map $f\in\Alg_G(K,k)$.

\begin{Theorem} There is an exact sequence
\begin{eqnarray*}
0 \to H^1_G(B,k) {\buildrel\pi^*\over\longrightarrow} \H^1_G(R,k) {\buildrel\kappa^*\over\longrightarrow} H^1_G(K,k) {\buildrel\delta\over\longrightarrow} H_G^2(B,k) {\buildrel(\pi\ot\pi)^*\over\longrightarrow} H_G^2(R,k)
\end{eqnarray*}
\end{Theorem}

\begin{proof}To construct $\del :H^1_G(K,k)\to H^2_G(B,k)$ observe first that
\begin{eqnarray*}
\Der_G(K,k)=H^1_G(K,k)=Z^1_G(K,k) =\{ f\in\Hom_G(K,k)|\del^1f=\del^0f + \del^2f\}.
\end{eqnarray*}
Choose a $K$-bimodule retraction $u:R\to K$ for $\kappa :K\to R$ so that $u\kappa = 1_K$ and $\ep_Ku=\ep_R$. Then $(\kappa\ot\kappa )^*\del^iu^*=\del^i\kappa^*u^*=\del^i$ for $i=0,1,2$. It also follows that
$$\del u^*f(K^+R\ot R + R\ot RK^+)=0$$
for any $f\in Z^1_G(K,k)$, since $u$ is a $K$-bimodule map, so that
\begin{eqnarray*}
\del u^*f(xr\ot r') & = & \ep (xr)u^*f(r') - u^*f(xrr') + u^*f(xr)\ep (r') \\
& = & -f(xu(rr')) + f(xu(r))\ep (r')  \\
& = & -f(x)\ep u(rr') + f(x)\ep u(r)\ep (r')=0
\end{eqnarray*}
and similarly $\del u^*f(r\ot r'x)=0$.
This means that the 2-cocycle $\del u^*f:R\ot R\to k$ factors uniquely through $\pi\ot\pi :R\ot R\to B\ot B$, i.e: there exists a unique 2-cocycle $\bar f:B\ot B\to k$ such that $(\pi\ot\pi )^*\bar f=\del u^*f$. So define
$$\delta :H^1_G(K,k)\to H^2_G(B,k)$$
by $\delta (f)=[\bar f]$, the cohomology class of $\bar f$.

Exactness at $H^1_G(B,k)$: It is clear that $\pi^*:H^1_G(B,k)\to H^1_G(R,k)$ is injective, since $\pi$ is surjective, and that $\kappa^*\pi^*=(\pi\kappa )^*=(\iota\ep )^*=\ep^*\iota^*$ is the trivial map, since $f(1)=0$ for $f\in Z^1_G(B,k)$.

Exactness at $H^1_G(R,k)$: Suppose that $f\in Z^1_G(R,k)$ and $\kappa^*(f)=0$. Then $f(RK^+R)=0$, since $f(rxr')=\ep (r)f(xr') + f(r)\ep (xr') = \ep (r)\ep (x)f(r')+\ep (r)f(x)\ep (r')+f(r)\ep (xr')=0$ for $x\in K^+$ . Hence, there is a unique $f'\in \Hom_G(B,k)$ such that $\pi^*(f')=f$. Moreover, $0=\del f=\del\pi^*f'=(\pi^*\ot\pi^*)\del f'$, so that $\del f'=0$, since $(\pi\ot\pi)^*$ is injective.

Exactness at $H^1(K,k)$: First show that $\delta\kappa^*=0$. If $f\in Z^1_G(R,k)$, then $f(1)=0$ and $\delta\kappa^*f=[\tilde f]\in H^2_G(B.k)$ with $\tilde f\in Z^2_G(B.k)$, and $(\pi\ot\pi )^*\tilde f=\del u^*\kappa^*f\in B^2_G(B,k)$ Moreover, $0=\del f(\kappa (x)\ot r)= -f(\kappa (x)r)+f\kappa (x)\ep (r)$, so that $(f-u^*\kappa^*f)(\kappa (x)r)=f(\kappa (x)r-f(\kappa (x)\kappa u(r))=f\kappa (x)\ep (r)-f\kappa (x)\ep\kappa u(r)=0$. Thus $f-u^*\kappa^*f$ factors uniquely through $\pi :R\to B$, i.e: $\pi^*f'=u^*\kappa^*f-f$ for a unique $f'\in\Hom_G(B,k)$. But then $(\pi\ot\pi )^*\del f'=\del\pi^*f'=\del(u^*\kappa^*f-f)-\del u^*\kappa^*f=(\pi\ot\pi )^*\tilde f$, and therefore $\tilde f=\del f'\in B^2_G(B,k)$, since $(\pi\ot\pi )^*$ is injective.

Now, if $\delta (f)=[\bar f] =0$ for a given $f\in Z^1_G(K,k)$, then
$\bar f=\del f'$ for some $f'\in\Hom_G(B^+,k)$, and $\del\pi^*f'= (\pi\ot\pi )^*\del f' = (\pi\ot\pi )^*\bar f = \del u^*f$, so that $u^*f - \pi^*f' \in Z^1_G(R,k)$. Then $\kappa^*(u^*f-\pi^*f')=f-(\pi\kappa )^*f'=f\in Z^1_G(K,k)$.

Exactness at $H^2_G(B,k)$: Finally, if $(\pi\ot\pi )^*[f]=0$ for a given $[f]\in H^2_G(B,k)$ then $(\pi\ot\pi )^*f=\del f'$ for some $f'\in\Hom_G(R^+,k)$. Moreover, $\del f'(\kappa (x)\ot r)=\ep\kappa (x)f'(r)-f'(\kappa (x)r)+f'\kappa (x)\ep (r) =(\pi\ot\pi )^*f(\kappa (x)\ot r)=f(\pi\kappa (x)\ot\pi (r))=f(\iota_B\ep_K(x)\ot\pi (r))$. Thus, if $x\in K^+$ then $f'(\kappa (x)r)=f'\kappa (x)\ep (r)$ and $(f'-u^*\kappa^*f')(\kappa (x)r)=f'(\kappa (x)r)-f'(\kappa (x)\kappa u(r))=f'\kappa (x)\ep (r)-f'\kappa (x)\ep\kappa u(r)=0$, so that $K^+R$ is in the kernel of $(f'-u^*\kappa^*f')$. It follows that there is a unique $f'':B\to k$ such that $\pi^*f''=f'-u^*\kappa^*f'$. Then $(\pi\ot\pi )^*(f-\del f'')=(\pi\ot\pi )^*f-\del\pi^*f''=\del u^*\kappa^*f'$, which means that $[f]=[f-\del f'']=\delta (\kappa^*f')$.
\end{proof}

The connecting map $\delta :H^1_G(K,k)\to H^2_G(B,k)$ can be used to describe the infinitesimal part of the \lq multiplicative' cocycles $\sigma :B\ot B\to k$ in terms of the algebra map $f\in\Alg_G(K,k)$ taking into account that $L(\mathcal D)\cong K^+/(K^+)^2$ and $H^1_G(K,k)\cong\Hom_G(K^+/(K^+)^2,k)$.

If $A=B\# kG$ then the collection of isomorphisms
$$\Psi_n:\Hom_G(B^n,k)\to {_G\Hom_G}(A^n,k),$$
given by $\Psi_n(f)(b_1g1, b_2g_2,\ldots , b_ng_n)=f(b_1, g_1(b_2),\ldots , g_1g_2\ldots g_{n-1}(b_n)$ and $\Psi^{-1}f'(b_1, b_2,\ldots , b_n)=f'(b_11, b_21,\ldots ,b_n1)$, defines an isomorphism of complexes, which induces an isomorphism in cohomology
$$\Psi^*:H^*_G(B,k)\to {_GH^*_G}(B\# kG,k).$$
The image of the composite
$$\Psi^2\delta :H^1_G(K,k)\to H^2_G(B,k)\to {_gH^2_G}(B\# kG,k)$$
consists of the infinitesimal parts of of the \lq multiplicative' cocycles. If $\zeta_F=\Psi^2\delta (f)$ then $(\zeta_f*m-m*\zeta_f)\in H^2(A,A)$ and $m+(\zeta_f*m-m*\zeta_f):A\ot A\to A$ is the infinitesimal part of the cocycle deformation associated with $f\in\Alg_G(K,k)$.

\section{Liftings and deformations}

The formal cocycle deformations of the previous section are in
particular formal deformations in the sense of \cite{GS, DCY}. The
subject of this section is the relation between formal
deformations, liftings and Hochschild cohomology of (braided) Hopf
algebras.

\subsection{Deformations of (graded) bialgebras} The formal deformation of a (graded) bialgebra $(A, m, \Delta ,\iota ,\ep )$
is a bialgebra structure $(A[[t]], m(t),\Delta (t), \iota ,\ep )$
on the free $k[[t]]$-module $A[[t]]=A\ot k[[t]]$, such that
$m(0)=m$ and $\Delta (0)=\Delta$. Here $m(t)=\sum{i\ge 0}\mu_it^i$
and $\Delta (t)=\sum_{i\ge 0}\delta_it^i$ are determined by
sequences of linear maps $\mu_i\colon A\ot A\to A$ and
$\delta_i\colon A\to A\ot A$. An $l$-deformation of $A$ is a
bialgebra structure on the free $k[[t]]/(t^{l+1})$-module
$A_l=A[[t]]/(t^{l+1})$. The associativity, coassociativity and
compatibility conditions are
\begin{enumerate}
\item Associativity: $\sum_{r+s=i}\mu_r(\mu_s\ot 1)=
\sum_{r+s=i}\mu_r(1\ot\mu_s)$,
\item Coassociativity: $\sum_{r+s=i}(\delta_r\ot 1)\delta_s=
\sum_{r+s=i}(1\ot\delta_r)\delta_s$,
\item Compatibility: $\sum_{r+s=i}\delta_r\mu_s =
\sum_{r+s+u+v=i}(\mu_r\ot\mu_s)\tau_{23}(\delta_u\ot\delta_v)$.
\end{enumerate}
In particular for infinitesimal deformations, the case $l=1$,
these are 2-cocycle conditions in the bialgebra cohomology.

An isomorphism of $l$-deformations is an isomorphism of
$k[[t]]/(t^{l+1})$-bialgebras
$$f\colon (A_l, m_l, \Delta_l )\to (A_l, m'_l, \Delta '_l)$$
such that $\iota^* (f)=id_A$. Such an isomorphism is of the
form $f=\sum_{i\ge 0}f_iT^i$ for a sequence of
maps $f_i\colon H\to H$ satisfying the conditions
\begin{itemize}
\item[(4)] $\sum_{r+s=i}f_r\mu_s =\sum_{t+u+v=i}\mu_t'(f_u\ot f_v)$
\item[(5)] $\sum_{r+s=i}\delta_r'f_s =\sum_{u+v+t=i}(f_u\ot f_v)\delta_t$
\end{itemize}
required by the fact that $f$ is a $k[[t]]/(t^{l+1})$-bialgebra map. The set of
isomorphism classes of $l$-deformations of $A$ will be denoted by
$\Def_l(A)$. The projection $k[[t]]/(t^{l+1})\to k[[t]]/(t^l$ induces a restriction
map $\res_l\colon \Def_{l+1}(A)\to\Def_l (A)$ and
$$\Def (A)=\lim_{\leftarrow}\Def_l(A)$$
is the set of isomorphism classes of formal deformations
($\infty$-deformations) of $A$.

\begin{Theorem} \cite{Gr2} The restriction map $\res_l\colon \Def_{l+1}(A)\to\Def_l (A)$
fits into an exact sequence of pointed sets
$$\begin{CD}
H^2(A,A) @> >> \Def_{l+1}(A) @>\res_l>> \Def_l(A) @>\obs_l>>  H^3(A,A)
\end{CD}$$
for $l\ge 0$. In particular:
\begin{itemize}
\item $H^2(A,A)\cong\Def_1(A)$ is an abelian group,
\item Every formal deformation of $A$ is trivial if and only if $H^2(A,A)=0$,
\item If $H^3(A,A)=0$ then every infinitesimal deformation can be
extended to a formal deformation.
\end{itemize}
\end{Theorem}

\begin{proof} (Sketch) If two $(l+1)$-deformations
restrict to the same $l$-deformation then they differ by a pair of
compatible 2-cocycles
$(\mu_{l+1}-\mu_{l+1}', \delta_{l+1}-\delta_{l+1}')$. If $(A_l, m_l,\Delta_l)$ is an $l$-deformation then
$$\psi=(\sum_{i+j=l+1}\mu_i(\mu_j\ot 1-1\ot\mu_j), \sum_{i+j=l+1}(\delta_i\ot 1-1\ot\delta_i)\delta_j)$$
is a 3-cocycle, and $\obs_l(A_l,m_l,\Delta_l)$ is the cohomology
class of $\psi$. Thus, if $\obs_l(A_l,m_l,\Delta_l)=0$ then $\psi
$ is a 3-coboundary, that is $\psi =\del
(\mu_{l+1},\delta_{l+1})=(\mu_0(1\ot\mu_{l+1}-\mu_{l+1}\ot
1)+\mu_{l+1}(1\ot\mu_0-\mu_0\ot 1),
(1\ot\delta_{l+1}-\delta_{l+1}\ot
1)\delta_0+(1\ot\delta_0-\delta_0\ot 1)\delta_{l+1})$ for some
$\mu_{l+1}\colon A\ot A\ot A\to A$ and $\delta_{l+1}\colon A\to
A\ot A\ot A$, and $(A_l, m_l, \Delta_l)=\res_l(A_{l+1}, m_{l+1},
\Delta_{l+1})$.
\end{proof}

\subsection{Liftings of (graded) bialgebras} If $A=\oplus_{n\ge 0}A_i$
is a graded bialgebra then it carries an
ascending bialgebra filtration $A_c^{(j)}=\oplus_{i\le j}A_i$ and
a descending bialgebra filtration $A^r_{(j)}=\oplus_{i\ge j}A_i$.
A lifting of a graded bialgebra $A$ is a filtered Hopf algebra
structure $K= (A, M, \Delta)$ on the vector space $A$ such that
$\gr_cK\cong A$. A co-lifting of $A$ is a co-filtered bialgebra
structure $G=(A, M, \Delta)$ such that $\gr_r(G)\cong A$. Let
$$\Lift (A) \quad \and \quad \co-Lift (A)$$
be the sets of equivalence classes of liftings and of co-liftings
of $A$, respectively.

\begin{Theorem}[cf. \cite{DCY}] There are bijections $\Lift (A)\cong\Def (A)\cong \co-Lift (A)$.
\end{Theorem}

\begin{proof} We deal with the co-lifting part of the theorem.
Let $G=(A, M, \Delta )$ be a co-lifting of the graded bialgebra
$H$, so that $\gr_r G= H$. The multiplication and the
comultiplication are maps of co-filtered vector spaces and they
uniquely determine maps $\mu_r\colon A\ot A\to A$ and
$\delta_r:A\to A\ot A$ of degree $r$ for every $r\ge 0$, such that
$M(a\ot b)=\sum_{r\ge 0}\mu_r(a\ot b)$ and $\Delta (c)=\sum_{r\ge
0}\delta_r(c)$. By associativity of $M$, coassociativity of
$\Delta$ and compatibility of the two structure maps these linear
maps satisfy exactly the conditions (1), (2) and (3) of the
previous subsection. Now define a bialgebra $D(G)=(A[[t]], m_d,
\Delta_d)$ over $k[[t]]$ by $m_d=\sum_{i\ge 0}\mu_it^i$ and $\Delta_d=\sum_{i\ge 0}\delta_it^i$. This gives a well-defined
bijection
$$D\colon \co-Lift(H)\to\Def(H)$$
since equivalent co-liftings are sent to isomorphic deformations.
An isomorphism of co-liftings $f\colon G\to G'$ is a map of
co-filtered bialgebras so that $f(a)=\sum_{r\ge 0}f_r(a)$ for
uniquely determined linear maps $f_r\colon A\to A$ of degree $r$,
which satisfy the conditions (4) and (5) of the previous
subsection since $f$ is a bialgebra map. The induced map
$f_d\colon D(G)\to D(G')$, defined by $f_d=\sum_{i\ge 0}f_iT^i$, is an isomorphism of deformations.
Similar arguments work for liftings \cite{DCY}, but now the linear
maps $\mu_r$, $\delta_r$ and $f_r$ are of degree $-r$.
\end{proof}

\begin{Lemma}[cf. \cite{GS}, \cite{MW}] If $\zeta\colon A\ot A\to k$ is a Hochschild cocycle of degree -n then the linear map
$\mu =(\zeta\ot m-m\ot\zeta )\Delta_{A\ot A}\colon A\ot A\to A$ is
of degree $-n$ and satisfies the cocycle condition
$$m(\mu\ot 1)+\mu (m\ot 1)=m(1\ot\mu )+\mu (1\ot m),$$
so that $m_{\zeta ,t}=m+\mu t^n\colon (A\ot A)[t]\to A[t]$ is an
infinitesimal deformation.
\end{Lemma}

\subsection{Examples}\label{ex1} To illustrate the discussion above let us consider the case of
one-dimensional crossed modules over a cyclic group.

1. Let $G=\langle g\rangle$ be a cyclic group of order $np$ and
let $V=kx$ be a 1-dimensional crossed $G$-module with action and
coaction given by $gx=qx$ for a primitive $n$-th root of unity $q$
and $\delta (x)=g\ot x$. The braiding $c\colon V\ot V\to V\ot V$
is then determined by $c(x\ot x)=qx\ot x$. The braided Hopf
algebra $\mathcal A(V)$ is the polynomial algebra $k[x]$ with
comultiplication $\Delta (x^i)=\sum_{r+s=i}{i\choose r}_q x^r\ot
x^s$ in which $x^n$ is primitive. The braided Hopf algebra
$\mathcal C(V)=k\langle x\rangle$ is the divided power Hopf
algebra with basis $\setst{x_i}{i\ge 0}$, comultiplication $\Delta
(x_i)=\sum_{r+s=i}x_r\ot x_s$ and multiplication
$x_ix_j={{i+j}\choose i}_qx_{i+j}$. The quantum symmetrizer
$\mathcal S:\mathcal A(V)\to\mathcal C(V)$ is given by $\mathcal S
(x^i)=\mathcal S(x)^i=i_q!x_i$. The Nichols algebra of $V$ is
$B(V)=\mathcal A(V)/(x^n)\cong\im\mathcal S$ and the Hopf algebra
\begin{eqnarray*}
A &=& B(V)\# k G \\
&=&\left\langle x, g|x^n=0, g^{np}=1, gx=qxg, \Delta (x)=x\ot
1+g\ot x, \Delta (g)=g\ot g\right\rangle
\end{eqnarray*}
is coradically graded.

The convolution invertible linear functional $\sigma\colon A\ot
A\to k$ defined by
$$\sigma (x^ig^u\ot x^jg^v)=\left\{ \begin{array}{ll} 1, & \mbox{if $i+j=0$}; \\
0, & \mbox{if $0 < i+j < n$};\\
aq^{ju}, & \mbox{if $i+j=n$}
\end{array}\right. $$
is a cocycle of the form $\sigma =\ep\ot\ep
+\zeta$, where $\zeta (x^ig^u\ot x^jg^v)= aq^{ju}\delta^{i+j}_n$
is a functional of degree $-n$ and $\zeta^2=0$, so that
$\sigma^{-1}=\ep\ot\ep -\zeta $. The resulting cocycle deformation
$$m_{\sigma}=(\sigma\ot m\ot\sigma^{-1})\Delta^{(2)}_{A\ot A}\colon A\ot A\to A$$
of the multiplication $m\colon A\ot A\to A$ is then given by
$$m_{\sigma}=m + (\zeta\ot m-m\ot\zeta )\Delta_{A\ot A}=\mu_0 +\mu_n$$
and is compatible with the original comultiplication. The explicit
expression of $m_{\sigma}$ in terms of the PBW-basis of $A$ is
$$m_{\sigma}(x^ig^j\ot x^kg^l)=
q^{jk}(x^{i+k}+ax^{\beta}(1-g^{n\alpha}))g^{j+l},$$ where
$i+j=n\alpha +\beta$ with $\alpha =0, 1$. The identity
$\sum_{s+v=\beta}{i\choose s}_q{k\choose
v}_qq^{s(k-v)}={{i+k}\choose \beta}$, which can be found in
\cite{Ka}, has been used in the calculations. The deformed Hopf
algebra has the presentation
$$A_{\sigma}=\left\langle\left. x,g\right| x^n=a(1-g^n), g^{np}=1, gx=qxg\right\rangle$$
with the original comultiplication, and since $\gr_cA_{\sigma}=A$
it is a lifting of $A$.

The linear dual $V^*=k\xi$, $\xi (x)=1$, is a crossed module over
the character group $\widehat{G}=\gen{\theta}$, $\theta
(g)=\alpha$ a primitive $np$-th root of unity and $\alpha^p=q$,
with action $\delta^*(\theta\ot\xi )=\theta\xi =\alpha\xi$ and
coaction $\mu^*(\xi )=\phi\ot\xi$, where $\phi =\theta^p$. The
graded braided Hopf algebra $\mathcal A(V^*)\cong k[\xi
]\cong\mathcal C(V)^*$ is the graded polynomial algebra with
comultiplication $\Delta (\xi^i )=\sum_{r+s=i}{i\choose r}_q
\xi^r\ot \xi^s$ so that $\xi^n$ is primitive. The cofree graded
braided Hopf algebra $\mathcal C(V^*)=k\gen{\xi} \cong\mathcal
A(V)^*$ is the divided power Hopf algebra with basis
$\setst{\xi_i}{i\ge 0}$, comultiplication $\Delta
(\xi_i)=\sum_{r+s=i}\xi_r\ot\xi_s$ and multiplication
$\xi_i\xi_j={{i+j}\choose i}_q\xi_{i+j}$. The quantum symmetrizer
$\mathcal S\colon \mathcal A(V^*)\to \mathcal C(V^*)$ is given by
$\mathcal S(\xi^i)=i_q!\xi_i$. The Nichols algebra of $V^*$ is
$B(V^*)=\mathcal A(V^*)/(\xi^n)\cong\im\mathcal S$ and the Hopf
algebra \begin{eqnarray*} A^*&=&B(V^*)\# k\widehat{G} \\
&=& \genst{\xi, \theta}{\xi^n=0, \theta^{np}=\ep , \theta \xi
=\alpha \xi\theta , \Delta (\xi )=\xi\ot\ep +\phi\ot\xi , \Delta
(\theta )=\theta\ot\theta } \end{eqnarray*} is radically graded.

The invertible element $\sigma^*\colon k\to A^*\ot A^*$ with
$\sigma^*(1)=\sigma =\ep\ot\ep
+\sum_{r+s=n}a_{rs}\xi^r\phi^s\ot\xi^s =\ep\ot\ep +\zeta$, with
$a_{rs}={1\over{r_q!s_q!}}$, is the cocycle above represented in
terms of the basis of $A^*$. Observe that $\zeta$ is of degree $n$
and $\zeta^2=0$. The resulting cocycle deformation of the
comultiplication
$$\Delta_{\sigma}=m^{(2)}_{A\ot A}(\sigma\ot\Delta\ot\sigma^{-1}=
\Delta +m_{A\ot A}(\zeta\ot\Delta -\Delta\ot\zeta
)=\delta_0+\delta_n,$$ where $m^{(2)}_{H\ot
H}(\zeta\ot\Delta\ot\zeta )=0$ is used, is compatible with the
original multiplication. Since $\Delta (\theta )\zeta
=\alpha^n\zeta\Delta (\theta )$, it follows that
$$\Delta_{\sigma}(\theta^i)=\theta^i\ot\theta^i +
(1-\alpha^{ni})\zeta (\theta^i\ot\theta^i).$$
Using the identity $a_{u-1,s}q^s+a_{u,s-1}=a_{u,s}$ one finds that
$\zeta\Delta (\xi )=\Delta(\xi )\zeta$, so that
$\Delta_{\sigma}(\xi )=\Delta (\xi )$ and
$$\Delta_{\sigma }(\xi^i\theta^j)=\Delta (\xi^i)\Delta_{\sigma}(\theta^j).$$
The deformed Hopf algebra has the presentation
$$A^{\sigma}=\genst{\xi , \theta}{\Delta_{\sigma}(\xi )=\Delta (\xi ),
\Delta_{\sigma}(\theta )= \theta\ot\theta +(1-\alpha^n)\zeta
(\theta\ot\theta )}$$ with the original multiplication and radical
filtration, so that $\gr_rA^{\sigma }=A^*$.

2. Let $G=\gen{g}$ be the cyclic group of order $np_1p_2$ and let
$\alpha$ be a primitive root of unity of order $np_1p_2$. Consider
the 1-dimensional crossed $G$-module $V=kx$ with action
$gx=\alpha^{p_2}x$ and coaction $\delta (x)=g^{p_1}\ot x$. The
braiding $c\colon V\ot V\to V\ot V$ is then given by $c(x\ot
x)=g^{p_1}x\ot x=\alpha^{p_1p_2}x\ot x$.

The dual space $V^*=k\xi$ is a crossed module over the character
group $\widehat{G} = \gen{\theta}$, where $\theta (g)=\alpha$. The
action is given by
$$\delta^*(\theta\ot\xi )(x)=(\theta\ot\xi )\delta (x)=
\theta(g^{p_1})\xi (x)=\alpha^{p_1}$$
and the coaction by
$$\mu^*(\xi (g^i\ot x)=\xi (g^ix)=\alpha^{ip_2}\xi (x)=\alpha^{ip_2},$$
so that $\delta^*(\theta\ot\xi )=\alpha^{p_1}\xi$ and $\mu^*(\xi
)=\theta^{p_2}\ot\xi $. The braiding map $c^*\colon V^*\ot V^*\to
V^*\ot V^*$ is determined by $c^*(\xi\ot\xi
)=\theta^{p_2}\xi\ot\xi =\alpha^{p_1p_2}\xi\ot\xi$. This means
that the dual $V^*$ is obtained essentially by interchanging the
role of $p_1$ and $p_2$, i.e: $(G, V, p_1, p_2)^*=(\widehat{G},
V^*, p_2, p_1)$.

The free graded braided Hopf algebra $\mathcal A(V)=k[x]$ is the
polynomial algebra with comultiplication $\Delta (x^i)=(x\ot
1+1\ot x)^i=\sum_{r+s=i}{i\choose r}_qx^r\ot x^s$, where
$q=\alpha^{p_1p_2}$. The ideal $(x^n)$ is a Hopf ideal, since
$x^n$ is primitive. The cofree graded braided Hopf algebra
$\mathcal C(V)=k\langle x\rangle$ is the divided power Hopf
algebra with basis $\setst{x_i}{i\ge 0}$, comultiplication $\Delta
(x_i)=\sum_{r+s=i}x_r\ot x_s$ and multiplication $m(x_i\ot
x_j)={{i+j}\choose i}_qx_{i+j}$. It follows in particular that
$x_1^i=i_q!x_i$ and $x_1^n=n_q!x_n=0$. The quantum symmetrizer
$\mathcal S\colon \mathcal A(V)\to\mathcal C(V)$ is determined by
$\mathcal S(x^i)=\mathcal S(x)^i=x_1^i=i_q!x_i$ and $\mathcal
S(x^n)=N_q!x_n=0$. The Nichols algebra of $V$ is then
$B(V)=\mathcal A(V)/(x^n)\cong\im\mathcal S\subset\mathcal C(V)$,
and $\im\mathcal S=\oplus_{i=0}^{n-1}kx_i$ is the Hopf subalgebra
of $\mathcal C(V)$ generated by $x_1$. The bosonization
$$A=B(V)\# kG =\genst{g, x}{g^{np_1p_2}=1, x^n=0, gx=\alpha^{p_2}xg,
\Delta (x)=x\ot 1+g^{p_1}\ot x}$$ is coradically as well as
radically graded, giving rise to liftings by deforming the
multiplication and co-liftings by deforming the comultiplication.

The linear functional $\zeta \colon A\ot A\to k$ of degree $-n$,
defined by $\zeta (x^ig^j\ot
x^kg^l)=\alpha^{p_2jk}\delta^{i+k}_n$, is a Hochschild cocycle
with $\zeta^2=0$, and satisfying
$$\zeta (m\ot 1)*(\zeta\ot\ep )=\zeta (1\ot m)*(\ep\ot\zeta ).$$
It follows that
$$\sigma =e^{\zeta}=\ep\ot\ep +\zeta\colon A\ot A\to k$$
is a convolution invertible multiplicative cocycle. In
terms of the dual basis of $A^*$ it can be expressed as
$$\sigma =\ep\ot\ep +\zeta =\ep\ot\ep +
\sum_{r+s=n}^{0<r,s<n}a_{rs}\xi^r\theta^{p_2s}\ot\xi^s ,$$
where $a_{rs}={1\over{r_q!s_q!}}$. The corresponding cocycle
deformation of the multiplication of $A$ is
$$m_{\sigma}=(\sigma\ot m\ot\sigma^{-1})\Delta^{(2)}_{A\ot A}=
m+(\zeta\ot m-m\ot\zeta )\Delta_{A\ot A},$$
since $(\zeta\ot m\ot\zeta )\delta_{A\ot A}=0$ (it is of degree $-2n$). Using
$$\Delta_{A\ot A}(x^ig^j\ot x^kg^l)=
\sum_{r+s=i}^{u+v=k}{i\choose r}_q{k\choose u}_qx^rg^{p_1s+j} \ot
x^ug^{p_1v+l}\ot x^sg^j\ot x^vg^l$$ and invoking the identity
\cite{Ka}
$$\sum_{s+v=\beta}{i\choose s}_q{k\choose v}_qq^{s(k-v)}=
{{i+k}\choose\beta }={{n+\beta }\choose\beta }=1$$ when
$i+k=n+\beta $, the following explicit formula for $m_{\sigma}$
can be deduced:
\begin{eqnarray*} &&\hskip -10pt
m_{\sigma}(x^ig^j\ot
x^kg^l)=\alpha^{p_2jk}x^{i+k}g^{j+l}\\
&&+\sum_{r+s=i}^{u+v=k}a{i\choose r}_q{k\choose
u}_q\alpha^{p_2((p_1s+j)u+jv)} (\delta^{r+u}_nx^{s+v} -
x^{r+u}\delta^{s+v}_ng^{p_1(s+v)})g^{j+l}\\
&&=\alpha^{p_2jk}[x^{i+k}+ax^{\beta} (\sum_{s+v=\beta}{i\choose
s}_q{k\choose v}_qq^{s(k-v)}\\
&&\hskip 10pt - \sum_{r+u=\beta}{i\choose
r}_q{k\choose u}_qq^{(i-r)u}g^{p_1n}]g^{j+l}\\
&&=\alpha^{p_2jk}[x^{i+k}+ax^{\beta}(1-g^{p_1n})]g^{j+l},
\end{eqnarray*}
where $i+k=n\gamma +\beta$ with $\gamma =0, 1$.

The element $\zeta\colon k\to A\ot A$, $\zeta (1)
=\sum_{r+s=n}^{0<r,s<n}a_{rs}x^rg^{p_1s}\ot x^s$ is a Hochschild
cocycle with $\zeta^2=0$. and satisfying
$$(\Delta\ot 1)(\zeta )(\zeta\ot 1)=(1\ot\Delta )(\zeta )(1\ot\zeta ).$$
It then follows that $\sigma\colon k\to A\ot A$, defined by
$$\sigma (1)=e^{\zeta}=1\ot 1 +\zeta =1\ot 1 +
\sum_{r+s=n}^{0<r,s<n}a_{rs}x^rg^{p_1s}\ot x^s,$$
is invertible and satisfies the multiplicative 2-cocycle condition
$$(\Delta\ot 1)(\sigma )(\sigma\ot 1)=(1\ot\Delta )(\sigma )(1\ot\sigma ).$$
The corresponding cocycle deformation of the comultiplication of $A$ is
\begin{eqnarray*}
\Delta_{\sigma}&=& m_{A\ot A}^{(2)}(\sigma\ot\Delta\ot\sigma^{-1}) \\
&=& \Delta +m_{H\ot H}(\zeta\ot\Delta -\Delta\ot\zeta ),
\end{eqnarray*}
since $m_{A\ot A}^{(2)}(\zeta\ot\Delta\ot\zeta )=0$
(it is of degree $2n$). In the resulting Hopf algebra $(A, m,
\Delta_{\sigma}, \iota, \ep )$ we have
$$\Delta_{\sigma}(g)= g\ot g +(1-\alpha^{p_2n})\zeta (g\ot g)$$
and $\Delta_{\sigma}(g^{p_1})=g^{p_1}\ot g^{p_1}=\Delta
(g^{p_1})$. Moreover, a simple calculation using the identity
$a_{u-1, s}q^s+a_{u. s-1} =a_{u, s}$ for $u+s=n+1$ shows that
$\zeta\Delta (x)=\Delta (x)\zeta $, so that
$$\Delta_{\sigma}(x)=\Delta (x)+m_{A\ot A}(\zeta\ot\Delta -
\Delta\ot\zeta )(x)=\Delta (x).$$

3. If $g=\gen{g}$ is the cyclic group of odd prime order $p$ and
$q$ is a primitive $p$-th root of unity, consider the
2-dimensional crossed $G$ module $V=kx_1\oplus kx_2$ with action
$gx_i=q^{(-1)^{i-1}}x_i$ and coaction $\delta (x_i)=g\ot x_i$. The
braiding map $c\colon V\ot V\to V\ot V$ is then $c(x_i\ot
x_j)=q^{(-1)^{j-1}}x_j\ot x_i$.

The dual space $V^*=k\xi_1\oplus k\xi_2$ is a crossed module over
the character group $\widehat{G}=\gen{\theta}$, where $\theta
(g)=q$, with action $\theta\xi_i=q\xi_i$, coaction
$\mu^*(\xi_i)=\theta^{(-1)^{i-1}}\ot\xi_i$ and braiding $c(x_i\ot
x_j)=q^{(-1)^{j-1}}\xi_j\ot\xi_i$.

The free graded braided Hopf algebra $\mathcal A(V)=T(V)$ is the
tensor algebra with comultiplication determined by $\Delta
(x_i)=x_i\ot 1+1\ot x_i$, where the braiding in the form of
$\Delta m=(m\ot m)(1\ot c\ot 1)(\Delta\ot\Delta )$ has to be taken
into account. In particular, \begin{eqnarray*}
\Delta(x^p_i)&=&(x_i\ot 1+1\ot x_i)^p=\sum_{r+s=p}{p\choose
r}_qx_i^r \ot x_i^s\\
&=&x_i^p\ot 1+1\ot x_i^p, \end{eqnarray*} since ${p\choose r}_q=0$
for $0<r<p$, and
\begin{eqnarray*}
\Delta ([x_1,x_2]_c)&=&[x_1,x_2]_c\ot 1+x_1\ot x_2 -c^2(x_1\ot
x_2) +1\ot [x_1,x_2]\\
&=& [x_1,x_2]_c\ot 1 +1\ot [x_1,x_2], \end{eqnarray*} since
$c^2(x_1\ot x_2)=x_1\ot x_2$, so that $[x_1,x_2]_c$ is primitive.
The cofree graded braided Hopf algebra $\mathcal C(V)=k\langle
x_1, x_2\rangle$ is the divided power Hopf algebra with basis all
words in the variables $\setst{x_i^{(r)}}{i=0,1; r\ge 0}$,
comultiplication $\Delta (x_i^{(l)}=\sum_{r+s=l}x_i^{(r)}\ot
x_i^{(s)}$ and multiplication $m(x_i^{(r)}\ot
x_i^{(s)})={{r+s}\choose r}_qx_i^{(r+s)}$ and $m(x_i^{1)}\ot
x_j^{(1)})=x_i^{(1)}x_j^{(1)}+q^{(-1)^{j-1}}x_j^{(1)}x_i^{(1)}$ if
$i\ne j$. It follows in particular that
$(x_i^{(1)})^r=r_q!x_i^{(r)}$, hence $(x_i^{(1)})^p=0$ and that
$[x_i^{(1)},x_j^{(1)}]_c=0$ if $i\ne j$. The quantum symmetrizer
$\mathcal S\colon \mathcal A(V)\to\mathcal C(V)$ is determined by
$\mathcal
S(x_ix_j)=x_i^{(1)}x_j^{(1)}+q^{(-1)^{j-1}}x_j^{(1})x_i^{(1)}$
$\mathcal S(x_i^r)=\mathcal S(x_i)^r=(x_i^{(1)})^r=r_q!x_i^{(1)}$,
so that $\mathcal S(x_i^p)=p_q!x_i^{(p)}=0$, and $\mathcal
S([x_i,x_j]_c)=0$ for $i\ne j$. The Nichols algebra of $V$ is then
$B(V)=\mathcal A(V)/(x_1^p, x_2^p. [x_1,x_2]_c)\cong\im\mathcal
S\subset\mathcal C(V)$, and $\im\mathcal S$ is the Hopf subalgebra
of $\mathcal C(V)$ generated  by $\set{x_1^{(1)}, x_2^{(1)} }$.
The bosonization
$$A=B(V)\# kG=\genst{g, X_1, x_2}{g^p=1, x_1^p=0, x_2^p=0, x_1x_2=q^{-1}x_2x_1}$$
with comultiplication $\Delta (g)=g\ot g$ and $\Delta (x_i)=x_i\ot
1+g\ot x_i$ is coradically graded.

The linear functional $\zeta\colon A\ot A\to k$ of degree $-2$,
defined by $\zeta (x_1^ix_2^jg^k\ot
x_1^rx_2^sg^t)=aq^k\delta^i_0\delta^j_1\delta^r_1\delta^s_0$, is a
Hochschild cocycle with $\zeta^p=0$

\section{Duals of pointed Hopf algebras}

\subsection{Liftings of Quantum linear spaces} The duals of finite
dimensional pointed Hopf algebras need not necessarily be pointed.
Although the dual of the bicross product $E=\cal B(V)\# kG$, which
is the bicross product $E^*=\cal B(W)\# k\widehat{G}$, is again
pointed, the duals of its liftings are generally not pointed. We
will explore the duals of such liftings $H$ when $V$ is a quantum
linear space over a finite abelian group $G$. Then
$V=\oplus_{i=1}^t kx_i$ with $x_i\in V_{g_i,\chi_i}$ with
$\chi_i(g_j)\chi_j(g_i)=1$. If $\chi_i(g_i)$ is a primitive
$n_i$-th root of unity then $\dim\cal B(V)=n_1n_2...n_t$ and by
\cite{Gr1} the finite dimensional liftings of $E$ are of the form
$$H(a)=\genst{G, V}{gx_i=\chi_i(g)x_ig,
[x_i,x_j]=a_{ij}(g_ig_j-1), x_i^{n_i}=a_{ii}(g^{n_i}-1)},$$ where
$a_{ij}=0$ when $g_ig_j=1$ or $\chi_i\chi_j\ne\epsilon$ for $i\ne
j$ and $a_{ii}=0$ when $g_i^{n_i}=1$ or $\chi_i^{n_i}\ne\epsilon$.
Let $G'$ be the subgroup of $G$ generated by $\setst{g_ig_j,
g_k^{n_k}}{a_{ij}\ne 0, a_{kk}\ne 0}$ and let $\bar G=G/G'$.
Observe that the sequence of character groups
$$1\to\widehat{G/G'}\to\widehat{G}\to\widehat{G'}\to 1$$
is exact, since $k^*$ is divisible ($k$ being algebraically closed).
Then
$$A=H(a)/\gen{[x_i,x_j], x_k^{n_k}}\cong\cal B(V)\# k\bar G$$ fits into
a commutative diagram
$$\begin{CD}
kG' @> >> kG @> >> kG/G' \\
@|             @V\kappa VV      @V\kappa VV \\
kG' @> >> H @> >> A
\end{CD}$$
The map $\pi\colon H\to kG$, $\pi (x^mg)=\delta_{0m}g$, is a
lifting of the canonical projection $\pi\colon A\to k\bar G$ and
obviously satisfies $\pi\kappa =1$. It is a coalgebra map, but not
an algebra map if $a\ne 0$. Dualizing we get $A^*=\cal B(V^*)\#
k\widehat{G/G'}$, a commutative diagram
$$\begin{CD}
A^* @> >> H^* @> >> k\widehat{G'} \\
@V\kappa^*VV @V\kappa^*VV @| \\
k\widehat{G/G'} @> >> k\widehat{G} @> >> k\widehat{G'}
\end{CD}$$
and an algebra section $\pi^*\colon k\widehat{G}\to H^*$ for
$\kappa^*$, $\pi^*(\chi )(x^mg)=\delta_{0m}\chi (g)$.

\begin{Proposition} Let $V=\oplus_{i=1}^tkx_i$ be a quantum linear space over
the finite abelian group $G$ with $\chi_i(g_i)$ of order $n_i$ in $k^*$, and
let $H$ be a non-trivial lifting of $E=\cal B(V)\# kG$. Then
\begin{enumerate}
\item $E^*\cong \cal B(V^*)\# k\widehat{G}$ and
\item $G(H^*)=\widehat{G/G'}$ is a proper subgroup of $\widehat{G}$ and $H^*$ is not
pointed.
\end{enumerate}
\end{Proposition}

\begin{proof} The dual $V^*=\oplus_{i=1}^tk\xi_i$,
where $\xi_i(x_j)=\delta_{ij}$, is a crossed $k\widehat{G}$-module
with action and coaction given by $\chi\xi_i=\chi (g_i)\xi_i$ and
$\delta (\xi_i)=\chi_i\ot\xi_i$. The commutativity of the diagram
$$\begin{CD}
\cal C(V)^* @>\cal S^*>> \cal A(V)^* \\
@V\cong VV  @V\cong VV \\
\cal A(V^*) @>\cal S >> \cal C(V^*)
\end{CD}$$
implies that $\cal B(V)^*\cong \cal B(V^*)$ as graded braided Hopf
algebras over $k\widehat{G}$.

If $\chi\in G(H^*)$, i.e: $\chi\colon H\to k$ is an algebra map,
then $\chi (g)\chi (x_i)=\chi (gx_i)=\chi_i(g)\chi
(x_ig)=\chi_i(g)\chi (x_i)\chi (g)$, hence $\chi (x_i)=0$ for all
$i$. This implies that $0=\chi (x_i^{n_i})=a_{ii}\chi
(g_i^{n_i}-1)$ for all $i$ and $0=\chi ([x_j,x_k])=a_{jk}\chi
(g_jg_k-1)$ for all $j<k$, so that $\chi (G')=1$. Thus
$G(H^*)\subseteq \widehat{G/G'}$, and since $\widehat{G/G'}=
G(A^*)\subseteq G(H^*)$, we conclude that $G(H^*)=\widehat{G/G'}$.
Now
$$\dim\Cor (H^*) =\dim (H/\Rad H)=\dim H-\dim\Rad H\ge
|G|=|\widehat{G}| > |G(H^*)|$$ implies that $\Cor H^*$ must
contain a non-trivial matrix coalgebra component, i.e: that $H^*$
is not pointed.

For a different proof observe that it suffices to show that $H/\Rad H$ is not
commutative, since then $\Cor (H^*)\cong (H/\Rad H)^*$ and hence $H^*$ is not
pointed. First observe that $kG\cap\Rad H=0$, since $\Rad H$ is nilpotent and
$\Rad kG=0$, so that
$$\begin{CD} kG @>\kappa >> H @>\eta >> H/\Rad H \end{CD}$$
is injective. If $H/\Rad H$ were commutative then
$$0=\eta (x_i)\eta (g)-\eta (g)\eta (x_i)=\eta (x_ig-gx_i)=(1-\chi_i(g))\eta (x_i)\eta (g)$$
for $1\le i\le t$ and every $g\in G$, and hence $x_i\in\Rad H$,
since $\chi_i\neq 1$. This would imply that
$x_i^{n_i}=a_{ii}(g_i^{n_i}-1)$ and $[x_j,x_k]=a_{ij}(g_jg_k-1)$
are in $\Rad H$ for all $i$ and all $j<k$, respectively,
contradicting $\Rad H\cap kG=0$.
\end{proof}

\subsection{Examples} Here are some examples of Hopf algebras with the property that $G(H^*)=\Alg (H,k)$ is trivial.
Let $G$ be a cyclic group of odd order $n$, $r>1$ a divisor of
$n$, $q$ a primitive $r$-th root of unity and $H=H(a)$ any lifting
of a quantum linear space over $G$ defined by the generators $g$,
$x$, $y$, the relations
$$g^n=1, gx=qxg, gy=q^{-1}yg, [x,y]=c(g^2-1), x^r=a(g^r-1), y^r=b(g^r-1)$$
and  comultiplication $\Delta (g)=g\ot g$, $\Delta x=x\ot 1+g\ot
x$, $\Delta (y)=y\ot 1+g\ot y$. If $c\ne 0$ then $G'=\gen{g^r,
g^2}=G$ and hence $G(H^*)\cong \widehat{G/G'}=\set{\varepsilon}$.

1. The examples in \cite{BDG} are of that form. If $p$ is an odd
prime number and $q$ is a primitive $p$-th root of unity, then the
Hopf algebra defined by generators $g$, $x$, $y$, relations
$$g^{p^2}=1, gx=qxg, gy=q^{-1}yq, x^p=a(g^p-1)=y^p, [x,y]=b(g^2-1)$$
and comultiplication $\Delta (g)=g\ot g$, $\Delta x=x\ot 1+g\ot
x$, $\Delta (y)=y\ot 1+g\ot y$. Then $\dim H=p^4$ and $G(H^*)$ is
trivial if $b\ne 0$.

2. If $p$ is an odd prime number and $q$ a primitive $p$-th root
of unity then the algebra defined by generators $g$, $x$, $y$ and
relations
$$g^p=1, gx=qxg, gy=q^{-1}yq, x^p=0=y^p, [x,y]=b(g^2-1)$$ is a Hopf
algebra with comultiplication $\Delta (g)=g\ot g$, $\Delta x=x\ot
1+g\ot x$, $\Delta (y)=y\ot 1+g\ot y$. Moreover, $\dim H=p^3$ and
$G(H^*)$ is trivial if $b\ne 0$.

This Hopf algebra also has an interesting property of having
exactly $p$ irreducible representations, one for each dimension
between $1$ and $p$. Indeed, assume that $G, X, Y$ are $r\times r$
matrices ($r\ge 2$, it is clear the $\varepsilon$ is the unique
$1$-dimensional representation), such that $g\mapsto G$, $x\mapsto
X$ and $y\mapsto Y$ is an irreducible representation. Note that,
if $e$ is an eigenvector for $G$ corresponding to an eigenvalue
$\psi$, then either $X^iY^j e=0$ or $X^iY^j e$ is an eigenvector
for $G$ corresponding to an eigenvalue $\psi\xi^{i-j}$. Now choose
an eigenvector $e$ for $G$ so that $Ye=0$ and note that vectors
$\mathcal{E}=(e,Xe,\ldots X^{r-1}e)$ must be a basis for $k^n$
(since $\langle G,X,Y\rangle=\bigvee\set{G^iX^jY^k}=M_r(k)$ and
$\bigvee X^i e=\bigvee G^iX^jY^ke$). In particular, this shows
that $r\le p$. In the ordered basis $\mathcal{E}$, the matrices
$G, X$ and $Y$ are as follows.

$$G=\psi
\begin{pmatrix}
  1 &     &    &   &   &   \\
    & \xi^{1} &    &   &   &   \\
    &     &  . &   &   &   \\
    &     &    & . &   &   \\
    &     &    &   & . &   \\
    &     &    &   &   & \xi^{r-l} \\
\end{pmatrix}\quad ,\quad
X=
\begin{pmatrix}

  0 &   &   &   &    &   \\

  1 & 0 &   &   &    &   \\

    & .  & . &   &    &   \\

    &   &  . & . &    &   \\

    &   &   &  . &  . &   \\

    &   &   &   &   1 & 0 \\

\end{pmatrix},$$
\vskip .5cm
$$Y=
\begin{pmatrix}

  0   & y_1      &    &   &         &   \\

      & 0     &  .  &   &         &   \\

      &       & .  & .  &         &   \\

      &       &    & . &  .       &   \\

      &       &    &   & .       & y_{r-1}  \\

      &       &    &   &         & 0 \\

\end{pmatrix}.$$

Now, using the identity $XY-\xi^{-1}YX=G^2-I$, it is easy to see
that $\psi^2=\xi^{1-r}$ and that
$$y_i=(\xi-\psi^2\xi^i)(\xi^i-1)(\xi-1)^{-1}=\xi (1-\xi^{i-r})(\xi^i-1)(\xi-1)^{-1}.$$
It is also straightforward to check, that $$(g,x,y)\mapsto
(G,X,Y),$$ where $G,X,Y$ are as above, is an irreducible
representation of $H$.

3. Here is an example of even dimension. Let $p$ be a prime number
and let $q$ be a primitive $p$-th root of unity. Define a Hopf
algebra by generators $g$, $x$, $y$, relations
$$g^{2p}=1, gx=qxg, gy=q^{-1}yq, x^p=a(g^p-1)=y^p, [x,y]=b(g^2-1)$$ and
comultiplication $\Delta (g)=g\ot g$,
$\Delta x=x\ot 1+g\ot x$, $\Delta (y)=y\ot 1+g\ot y$.
Then $\dim H=2p^3$ and $G(H^*)$ is trivial if $a\ne 0\ne b$.

4. Let $n=rs$ be a positive integer with $\gcd (r,s)=1$. If
$G=\gen{g}$ is a cyclic group of order $n$, then the character
group $\widehat{G}=\gen{\chi}$ is also cyclic of order $n$. Let
$H$ be the Hopf algebra defined as an algebra by the generators
$g$, $x$, $y$ and the relations
\begin{eqnarray*}
g^n &=&1, gx=\chi^r(g)x g, gy=\chi^{-r}(g)y g,\\
x^r&=& a(g^s-1), y^r=c(g^s-1), [x,y]=b(g^2-1), \end{eqnarray*}
with comultiplication $\Delta (g)=g\ot g$, $\Delta x=x\ot 1+g\ot
x$, $\Delta (y)=y\ot 1+g\ot y$. Then $\chi^r(g)$ and $\chi^{-r}(g)$
have order $s$. Moreover, $\dim H=ns^2$. Since $\gcd (r,s)=1$, it
follows that $G_0=\gen{g^s, g^2}$ if $a\ne 0\ne c$, even if $b=0$.
Thus $G(H^*)$ is trivial, if $s$ is odd, and of order $2$ if $s$
is even.

\begin{appendix}
\end{appendix}

\end{document}